\documentclass{scrartcl}

\usepackage[utf8]{luainputenc}
\usepackage[UKenglish]{babel}
\usepackage{csquotes}

\usepackage[a4paper,top=27mm,bottom=20mm,inner=25mm,outer=20mm]{geometry}
\usepackage[plainpages=false,pdfpagelabels,hidelinks,unicode]{hyperref}

\usepackage[%
  backend=biber,
  style=numeric-comp,
  giveninits=true, uniquename=init, %abbreviate first names
  natbib=true,
  url=true,
  doi=true,
  isbn=false,
  backref=false,
  maxnames=99,
  ]{biblatex}
\usepackage{amsmath}
\allowdisplaybreaks
\usepackage{amssymb}
\usepackage{commath}
\usepackage{mathtools}
\usepackage{bbm}
\usepackage{nicefrac}

\usepackage{siunitx}

\usepackage{amsthm}
\theoremstyle{plain}
  \newtheorem{theorem}{Theorem}
  \newtheorem{lemma}[theorem]{Lemma}
  \newtheorem{proposition}[theorem]{Proposition}
  
  \newtheorem{corollary}[theorem]{Corollary}
\theoremstyle{definition}
  \newtheorem{definition}[theorem]{Definition}
  \newtheorem{remark}[theorem]{Remark}
  \newtheorem{example}[theorem]{Example}
\numberwithin{theorem}{section}

\usepackage{color}
\usepackage{graphicx}
\usepackage[small]{caption}
\usepackage{subcaption}

\begingroup\expandafter\expandafter\expandafter\endgroup
\expandafter\ifx\csname pdfsuppresswarningpagegroup\endcsname\relax
\else
\fi

\usepackage{booktabs}
\usepackage{rotating}
\usepackage{multirow}

\usepackage{enumitem}

\usepackage{ifluatex}
\ifluatex
  \usepackage[no-math]{fontspec}
\else
  \usepackage[T1]{fontenc}
\fi
\usepackage{newpxtext,newpxmath}

\newcommand{\N}{\mathbb{N}}
\newcommand{\R}{\mathbb{R}}
\newcommand{\scp}[2]{\left\langle{#1,\, #2}\right\rangle}
\newcommand{\I}{\operatorname{I}}

\DeclarePairedDelimiterX\newset[1]\lbrace\rbrace{\setaux #1||\endsetaux}
\def\setaux#1|#2|#3\endsetaux{\if\relax\detokenize{#2}\relax #1 \else #1 \;\delimsize\vert\; #2 \fi}
\renewcommand{\set}[1]{\newset*{#1}}

\let\epsilon\varepsilon
\let\phi\varphi
\let\rho\varrho

\renewcommand{\div}{\operatorname{div}}
\newcommand{\rot}{\operatorname{rot}}
\newcommand{\curl}{\operatorname{curl}}
\newcommand{\grad}{\operatorname{grad}}

\newcommand{\kernel}{\operatorname{ker}}
\newcommand{\image}{\operatorname{im}}
\newcommand{\spann}{\operatorname{span}}

\newcommand{\st}{\operatorname{s.t.}\;}
\newcommand{\argmin}{\operatorname{arg~min}}

\newcommand{\osc}{\vec{\mathrm{osc}}}
\renewcommand{\vec}[1]{\pmb{#1}}

\newenvironment{keywords}{\par\textbf{Key words.}}{\par}
\newenvironment{AMS}{\par\textbf{AMS subject classification.}}{\par}

\title{Discrete Vector Calculus and Helmholtz Hodge Decomposition for Classical
       Finite Difference Summation by Parts Operators}
\author{Hendrik Ranocha, Katharina Ostaszewski, Philip Heinisch}
\date{1st December 2019}

\makeatletter
\hypersetup{pdfauthor={\@author}}
\hypersetup{pdftitle={\@title}}
\makeatother

\begin{document}

\maketitle

\begin{abstract}
In this article, discrete variants of several results from vector calculus are
studied for classical finite difference summation by parts operators in two and
three space dimensions. It is shown that existence theorems for scalar/vector
potentials of irrotational/solenoidal vector fields cannot hold discretely because
of grid oscillations, which are characterised explicitly. This results in a
non-vanishing remainder associated to grid oscillations in the discrete Helmholtz
Hodge decomposition. Nevertheless, iterative numerical methods based on an interpretation
of the Helmholtz Hodge decomposition via orthogonal projections are proposed
and applied successfully. In numerical experiments, the discrete remainder vanishes
and the potentials converge with the same order of accuracy as usual in other first
order partial differential equations. Motivated by the successful application of
the Helmholtz Hodge decomposition in theoretical plasma physics, applications to
the discrete analysis of magnetohydrodynamic (MHD) wave modes are presented and
discussed.

\end{abstract}



\section{Introduction}

The Helmholtz Hodge decomposition of a vector field into irrotational and solenoidal
components and their respective scalar and vector potentials is a classical result
that appears in many different variants both in the traditional fields of mathematics
and physics and more recently in applied sciences like medical imaging \cite{SIMS2018}.
Especially in the context of classical electromagnetism and plasma physics, the Helmholtz
Hodge decomposition has been used for many years to help analyse turbulent velocity fields
\cite{Kowal2010,Beresnyak2019} or separate current systems into source-free and
irrotational components \cite{KHG1983,KHG1984,KHG1988,KHG1999}.
Numerical implementations can be useful for different tasks, as described in the
survey article \cite{bhatia2012helmholtz} and references cited therein. Some
recent publications concerned with (discrete) Helmholtz Hodge decompositions are
\cite{angot2013fast,ahusborde2014discrete,lemoine2015discrete}.

The main motivation for this article is the analysis of numerical solutions of
hyperbolic balance/conservation laws such as the (ideal) magnetohydrodynamic (MHD)
equations. Since the Helmholtz Hodge decomposition is a classical tool for the
(theoretical) analysis of these systems, it is reasonable to assume that it can
be applied fruitfully also in the discrete context.

For the hyperbolic partial differential equations of interest, summation by parts
(SBP) operators provide a means to create stable and conservative discretisations
mimicking energy and entropy estimates available at the continuous level, cf.
\cite{svard2014review,fernandez2014review} and references cited therein. While
SBP operators originate in the finite difference (FD) community
\cite{kreiss1974finite,strand1994summation}, they include also other schemes such as finite volume (FV)
\cite{nordstrom2001finite,nordstrom2003finite}, discontinuous Galerkin (DG)
\cite{gassner2013skew,fernandez2014generalized}, and the recent flux
reconstruction/correction procedure via reconstruction schemes
\cite{huynh2007flux,ranocha2016summation}.

SBP operators are constructed to mimic integration by parts discretely. Such
mimetic properties of discretisations can be very useful to transfer results
from the continuous level to the discrete one and have been of interest in
various forms \cite{hyman1997natural,lipnikov2014mimetic,ranocha2019mimetic}.
In this article, the focus will lie
on finite difference operators, in particular on nullspace consistent ones.
Similarly, (global) spectral methods based on Lobatto Legendre nodes can also be
used since they satisfy the same assumptions.

Several codes applied in practice for relevant problems of fluid mechanics
or plasma physics are based on collocated SBP operators.
To the authors' knowledge, the most widespread form of SBP operators used in practice
are such collocated ones \cite{svard2014review,fernandez2014review}. These are
easy to apply and a lot of effort has gone into optimising such schemes
\cite{mattsson2014optimal,mattsson2018boundary}.
Since this study is motivated by the analysis of numerical results
obtained using widespread SBP operators resulting in provably stable
methods, it is natural to consider classical collocated SBP operators.
The novel numerical analysis provided in this article shows that a
discrete Helmholtz Hodge decomposition and classical representation
theorems of divergence/curl free vector fields cannot hold discretely
in this setting. Even such a negative result is an important contribution,
not least since the reason for the failure of these theorems from vector
calculus for classical SBP operators is analysed in detail, showing the
importance of explicitly characterised grid oscillations. Based on this
analysis, numerical evidence is reported, showing that the negative
influence of grid oscillations vanishes under grid refinement for smooth
data. Moreover, means are developed to cope with the presence of such
adversarial grid oscillations.

This article is structured as follows. Firstly, the concept of summation by parts
operators is briefly reviewed in Section~\ref{sec:SBP}. Thereafter, classical
existence theorems for scalar/vector potentials of irrotational/solenoidal vector
fields are studied in the discrete context in Section~\ref{sec:potentials}.
It will be shown that these representation theorems cannot hold discretely.
Furthermore, the kernels of the discrete curl and divergence operators will be
characterised via the images of the discrete gradient and curl operators and some
additional types of grid oscillations. After a short excursion to the discrete
characterisation of vector fields that are both divergence and curl free as
gradients of harmonic functions in Section~\ref{sec:div-curl-free}, the discrete
Helmholtz decomposition is studied in Section~\ref{sec:helmholtz}. It will be
shown that classical SBP operators cannot mimic the Helmholtz Hodge decomposition
$u = \grad \phi + \curl v$ discretely. Instead, the remainder $r = u - \grad \phi
- \curl v$ will in general not vanish in the discrete setting. Nevertheless, this
remainder in associated with certain grid oscillations and converges to zero,
as shown in numerical experiments in Section~\ref{sec:numerical-examples}.
Additionally, applications to the analysis of MHD wave modes are presented and
discussed.
Finally, the results are summed up and discussed in Section~\ref{sec:summary}
and several directions of further research are described.

\section{Summation by Parts Operators}
\label{sec:SBP}

In the following, finite difference methods on Cartesian grids will be used.
Hence, the one dimensional setting is described at first.
Throughout this article, functions at the continuous level are denoted using
standard font, e.g. $u$, and discrete grid functions are denoted using
boldface font, e.g. $\vec{u}$, independently on whether they are scalar
or vector valued.

The given domain $\Omega = [x_L, x_R]$ is discretised using a uniform grid with
nodes $x_L = x_1 < x_2 < \dots < x_N = x_R$ and a function $u$ on $\Omega$ is
represented discretely as a vector $(\vec{u}^{(a)})_a$, where the components are the
values at the grid nodes, i.e. $\vec{u}^{(a)} = u(x_a)$. Since a collocation setting is
used, the grid is the same for every (vector or scalar valued) function and both
linear and nonlinear operations are performed componentwise. For example, the
product of two functions $u$ and $v$ is represented by the Hadamard product of
the corresponding vectors, i.e. $(\vec{u} \vec{v})^{(a)} = \vec{u}^{(a)} \vec{v}^{(a)}$.

\begin{definition}
  An SBP operator with order of accuracy $p \in \N$ on $\Omega = [x_L,x_R] \subset \R$
  consists of the following components.
  \begin{itemize}
    \item
    A discrete derivative operator $D$, approximating the derivative $\partial_x u$
    as $D \vec{u}$ with order of accuracy $p$.

    \item
    A symmetric and positive definite mass matrix\footnote{The name ``mass matrix''
    is common for finite element methods such as discontinuous Galerkin methods,
    while ``norm matrix'' is more common in the finite difference community. Here,
    both names will be used equivalently.} $M$, approximating the
    scalar product on $L^2(\Omega)$ via
    \begin{equation}
      \vec{u}^T M \vec{v}
      =
      \scp{\vec{u}}{\vec{v}}_M
      \approx
      \scp{u}{v}_{L^2(\Omega)}
      =
      \int_\Omega u \cdot v.
    \end{equation}

    \item
    A boundary operator $E$, approximating the difference of boundary values as
    in the fundamental theorem of calculus as $u(x_R) v(x_R) - u(x_L) v(x_L)$
    via $\vec{u}^T E \vec{v}$ with order of accuracy $p$.

    \item
    Finally, the SBP property
    \begin{equation}
    \label{eq:SBP}
      M D + D^T M = E
    \end{equation}
    has to be fulfilled.
  \end{itemize}
\end{definition}

The SBP property \eqref{eq:SBP} ensures that integration by parts is mimicked
discretely as
\begin{gather}
\label{eq:SBP-IBP}
  \begin{array}{ccc}%{*3{>{\displaystyle}c}}
    \underbrace{
      \vec{u}^T M D \vec{v}
      + \vec{u}^T D^T M \vec{v}
    }
    & = &
    \underbrace{
      \vec{u}^T E \vec{v},
    }
    \\
    \rotatebox{90}{$\!\approx\;$}
    &&
    \rotatebox{90}{$\!\!\approx\;$}
    \\
    \overbrace{
      \int_{x_L}^{x_R} u \, (\partial_x v)
      + \int_{x_L}^{x_R} (\partial_x u) \, v
    }
    & = &
    \overbrace{
      u \, v \big|_{x_L}^{x_R}
    },
  \end{array}
\end{gather}

In the following, finite difference operators on nodes including the boundary
points will be used. In that case, $E = \operatorname{diag}(-1,0,\dots,0,1)$.

For the numerical tests, only diagonal norm SBP operators are considered, i.e.
those SBP operators with diagonal mass matrices $M$, because of their improved
properties for (semi-) discretisations
\cite{svard2004coordinate,sjogreen2017skew,fisher2013high}.
In this case, discrete integrals are evaluated using the quadrature provided by
the weights of the diagonal mass matrix \cite{hicken2013summation}. While there
are also positive results for dense norm operators, the required techniques are
more involved
\cite{ranocha2017extended,ranocha2017shallow,chan2018discretely,fernandez2019extension}.
However, the techniques and results of this article do not depend on diagonal
mass matrices.

For classical diagonal norm SBP operators, the order of accuracy is $2p$ in the
interior and $p$ at the boundaries \cite{kreiss1974finite,linders2018order},
allowing a global convergence order of $p+1$ for hyperbolic problems
\cite{svard2006order,svard2014note,svard2017convergence}. Here, SBP operators
will be referred to by their interior order of accuracy $2p$.

\begin{example}
\label{ex:SBP-2}
  The classical second order accurate SBP operators are
  \begin{equation}
  \label{eq:SBP-2}
    D
    =
    \frac{1}{2 \Delta x}
    \begin{pmatrix}
      -2 & 2 \\
      -1 & 0 & 1 \\
      & \ddots & \ddots & \ddots \\
      && -1 & 0 & 1 \\
      &&& -2 & 2
    \end{pmatrix},
    \qquad
    M
    =
    \Delta x
    \begin{pmatrix}
      \frac{1}{2} \\
      & 1 \\
      && \ddots \\
      &&& 1 \\
      &&&& \frac{1}{2}
    \end{pmatrix},
  \end{equation}
  where $\Delta x$ is the grid spacing.
  Thus, the first derivative is given by the standard second order central derivative
  in the interior and by one sided derivative approximations at the boundaries.
\end{example}

SBP operators are designed to mimic the basic integral theorems of vector calculus
(fundamental theorem of calculus, Gauss' theorem, Stokes' theorem) in the given
domain $\Omega$ (but not necessarily on subdomains of $\Omega$). However, this
mimetic property does not suffice for the derivations involving scalar and vector
potentials in the following. Hence, nullspace consistency will be used as
additional mimetic property that has also been used in
\cite{svard2017convergence,linders2019convergence}.
\begin{definition}
  An SBP derivative operator $D$ is nullspace consistent, if the nullspace/kernel
  of $D$ is $\kernel D = \spann \set{\vec{1}}$.
\end{definition}
\begin{remark}
  A consistent derivative operator $D$ satisfies
  $\spann \set{\vec{1}} \subseteq \kernel D$.
  Some undesired behaviour can occur if $\kernel D \neq \spann \set{\vec{1}}$, cf.
  \cite{svard2017convergence,linders2019convergence,ranocha2019some}.
\end{remark}

In multiple space dimensions, tensor product operators will be used, i.e. the one
dimensional SBP operators are applied accordingly in each dimension. In the following,
$\I_{s}$, $s \in \set{x,y,z}$, are identity matrices and $D_{s}, M_{s}, E_{s}$,
$s \in \set{x,y,z}$, are one dimensional SBP operators in the corresponding
coordinate directions.

\begin{definition}
  In two space dimensions, the tensor product operators are
  \begin{equation}
  \begin{gathered}
  \begin{aligned}
    D_1 &= D_x \otimes \I_y, &
    D_2 &= \I_x \otimes D_y,
    \\
    E_1 &= E_x \otimes M_y, &
    E_2 &= M_x \otimes E_y,
  \end{aligned}
    \\
    M = M_x \otimes M_y,
  \end{gathered}
  \end{equation}
  and the vector calculus operators are
  \begin{equation}
    \grad = \begin{pmatrix} D_1 \\ D_2 \end{pmatrix}, \quad
    \rot = \begin{pmatrix} D_2 \\ -D_1 \end{pmatrix}, \quad
    \curl = \begin{pmatrix} -D_2, D_1 \end{pmatrix}, \quad
    \div = \begin{pmatrix} D_1, D_2 \end{pmatrix}.
  \end{equation}
\end{definition}
\begin{remark}
  In two space dimensions, $\curl$ maps vector fields to scalar fields and
  $\rot$ maps scalar fields to vector fields. In three space dimensions, both
  operations correspond to the classical curl of a vector field.
\end{remark}

\begin{definition}
  In three space dimensions, the tensor product operators are
  \begin{equation}
  \begin{gathered}
  \begin{aligned}
    D_1 &= D_x \otimes \I_y \otimes \I_z, &
    D_2 &= \I_x \otimes D_y \otimes \I_z, &
    D_3 &= \I_x \otimes \I_y \otimes D_z,
    \\
    E_1 &= E_x \otimes M_y \otimes M_z, &
    E_2 &= M_x \otimes E_y \otimes M_z, &
    E_3 &= M_x \otimes M_y \otimes E_z.
  \end{aligned}
    \\
    M = M_x \otimes M_y \otimes M_z,
  \end{gathered}
  \end{equation}
  and the vector calculus operators are
  \begin{equation}
    \grad = \begin{pmatrix} D_1 \\ D_2 \\ D_3 \end{pmatrix}, \quad
    \curl = \begin{pmatrix} 0 & -D_3 & D_2 \\ D_3 & 0 & -D_1 \\ -D_2 & D_1 & 0 \end{pmatrix}, \quad
    \div = \begin{pmatrix} D_1, D_2, D_3 \end{pmatrix}.
  \end{equation}
\end{definition}

\begin{remark}
\label{rem:discrete-derivatives-commute}
  The standard tensor product discretisations of vector calculus operators given
  above satisfy $\div \curl = 0$ (or $\div \rot = 0$) and $\curl \grad = 0$,
  since the discrete derivative operators commute, i.e. $D_j D_i = D_i D_j$.
\end{remark}

\section{Scalar and Vector Potentials}
\label{sec:potentials}

Classical results of vector calculus in three space dimensions state that (under
suitable assumptions on the regularity of the domain and the vector fields)
\begin{itemize}
  \item
  a vector field $u$ is curl free if and only if $u$ has a scalar potential
  $\phi$, i.e. $u = \grad \phi$,

  \item
  a vector field $u$ is divergence free if and only if $u$ has a vector potential
  $v$, i.e. $u = \curl v$.
\end{itemize}
Using modern notation, these classical theorems can be formulated as follows,
cf. \cite[Theorem~I.2.9]{girault1986finite}, \cite[Corollary~2]{schweizer2018friedrichs}
and \cite[Lemma~4.4]{jikov1994homogenization}. In the following, $\Omega$ is
always assumed to be a bounded rectangle/cuboid in $\R^d$, $d \in \set{2,3}$.
\begin{theorem}
\label{thm:continuous-scalar-potential}
  A vector field $u \in L^2(\Omega)^d$ satisfies $\curl u = 0$ if and only if
  there exists a scalar potential $\phi \in H^1(\Omega)$ satisfying $u = \grad \phi$.
\end{theorem}

\begin{theorem}
\label{thm:continuous-vector-potential}
  A vector field $u \in L^2(\Omega)^d$ satisfies $\div u = 0$ if and only if
  \begin{itemize}
    \item
    there exists a potential $v \in H^1(\Omega)$ satisfying $u = \rot v$
    if $d = 2$,

    \item
    there exists a vector potential $v \in H(\curl; \Omega)$ satisfying $u = \curl v$
    if $d = 3$.
  \end{itemize}
  Here, the rotation/curl of a scalar field $v$ is defined as
  $\rot v = (\partial_2 v, -\partial_1 v)$.
\end{theorem}

In the rest of this section, discrete analogues of these theorems will be studied
for SBP operators. Before theorems characterising the discrete case can be proved,
some preliminary results have to be obtained at first.

\subsection{Grid Oscillations}

For a nullspace consistent SBP derivative operator $D$, the kernel of its adjoint
operator $D^* = M^{-1} D^T M$ will play an important role in the following.
\begin{lemma}
  For a nullspace consistent SBP derivative operator $D$ in one space dimension,
  $\dim \kernel D^* = 1$.
\end{lemma}
\begin{proof}
  For $N$ grid nodes,
  $\dim \kernel D^* = \dim (\image D)^\perp = N - \dim \image D = \dim \kernel D
  = \dim \spann \set{\vec{1}} = 1$.
\end{proof}
\begin{definition}
  A fixed but arbitrarily chosen basis vector of $\kernel D^*$ for a nullspace
  consistent SBP operator $D$ is denoted as $\osc$. In two space dimensions,
  \begin{equation}
  \begin{gathered}
  \begin{aligned}
    \osc_1 &= \osc_x \otimes \vec{1}, &
    \osc_2 &= \vec{1} \otimes \osc_y,
  \end{aligned}
    \\
    \osc_{12} = \osc_x \otimes \osc_y,
  \end{gathered}
  \end{equation}
  and in three space dimensions
  \begin{equation}
  \begin{gathered}
  \begin{aligned}
    \osc_1 &= \osc_x \otimes \vec{1} \otimes \vec{1}, &
    \osc_2 &= \vec{1} \otimes \osc_y \otimes \vec{1}, &
    \osc_3 &= \vec{1} \otimes \vec{1} \otimes \osc_z,
    \\
    \osc_{12} &= \osc_x \otimes \osc_y \otimes \vec{1}, &
    \osc_{13} &= \osc_x \otimes \vec{1} \otimes \osc_z, &
    \osc_{23} &= \vec{1} \otimes \osc_y \otimes \osc_z,
  \end{aligned}
    \\
    \osc_{123} = \osc_x \otimes \osc_y \otimes \osc_z.
  \end{gathered}
  \end{equation}
\end{definition}

The name $\osc$ shall remind of (grid) oscillations, since the kernel of $D^*$ is
orthogonal to the image of $D$ which contains all sufficiently resolved functions.
\begin{example}
\label{ex:SBP-2-osc}
  For the classical second order SBP operator of Example~\ref{ex:SBP-2},
  \begin{equation}
    D^*
    =
    M^{-1} D^T M
    =
    \frac{1}{2 \Delta x}
    \begin{pmatrix}
      -2 & -2 \\
      1 & 0 & -1 \\
      & \ddots & \ddots & \ddots \\
      && 1 & 0 & -1 \\
      &&& 2 & 2
    \end{pmatrix} \in \R^{N \times N},
  \end{equation}
  and $\kernel D^* = \spann \set{\osc}$, where
  \begin{equation}
    \begin{cases}
      \osc^{(1)} = \osc^{(3)} = \dots = \osc^{(N)} = -\osc^{(2)} = -\osc^{(4)} = \dots = -\osc^{(N-1)},
      & N \text{ odd},
      \\
      \osc^{(1)} = \osc^{(3)} = \dots = \osc^{(N-1)} = -\osc^{(2)} = -\osc^{(4)} = \dots = -\osc^{(N)},
      & N \text{ even}.
    \end{cases}
  \end{equation}
  Thus, $\osc$ represents classical grid oscillations. Grid oscillations for the
  SBP derivative operators of \cite{mattsson2004summation} are visualised in
  Figure~\ref{fig:grid_oscillations}. These grid oscillations alternate between
  $+1$ and $-1$ in the interior of the domain. Near the boundaries, the values
  depend on the order and boundary closure of the scheme.
\end{example}

\begin{figure}[ht]
\centering
  \begin{subfigure}{0.49\textwidth}
    \centering
    \includegraphics[width=\textwidth]{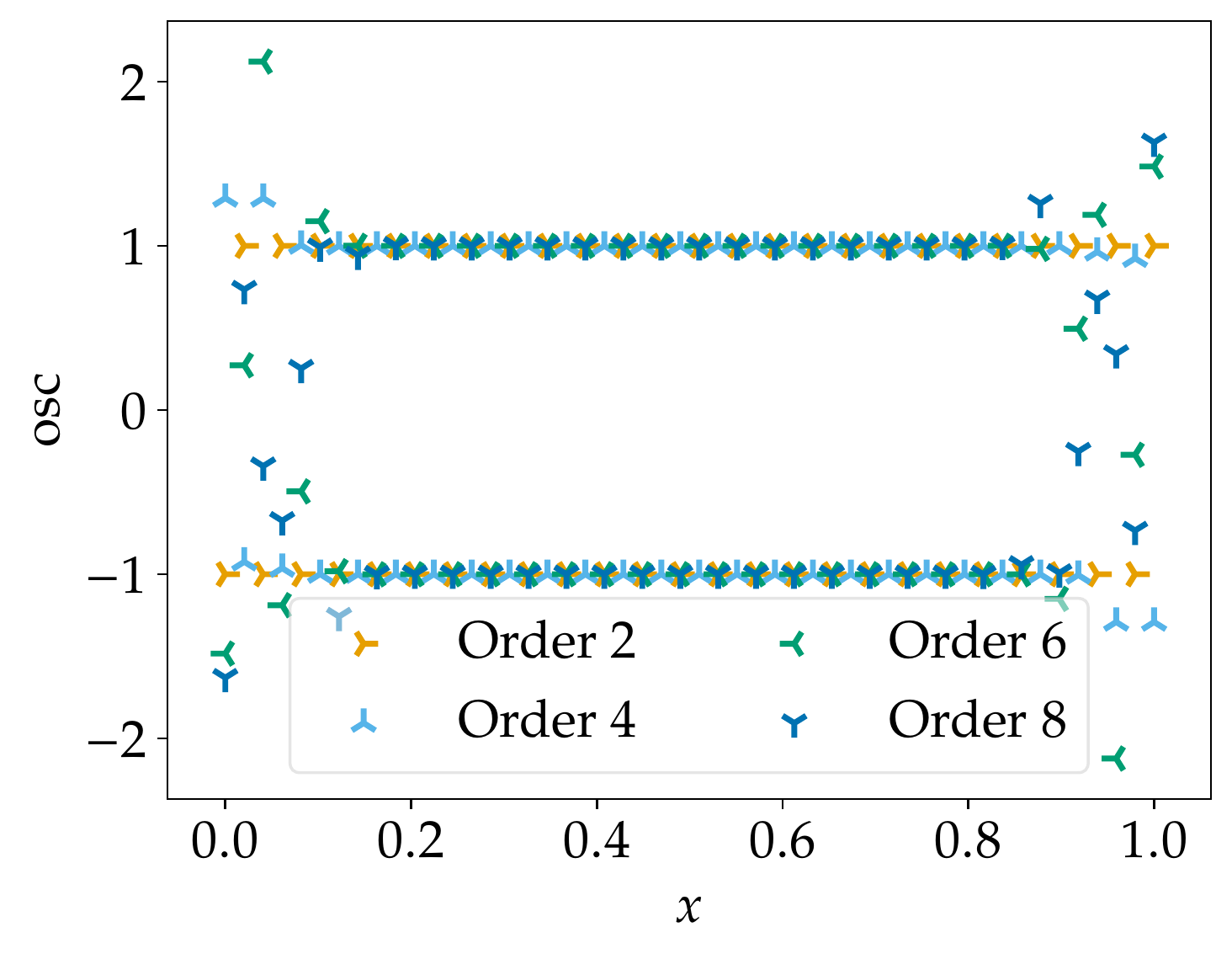}
    \caption{$N = 50$ grid points.}
  \end{subfigure}%
  \begin{subfigure}{0.49\textwidth}
    \centering
    \includegraphics[width=\textwidth]{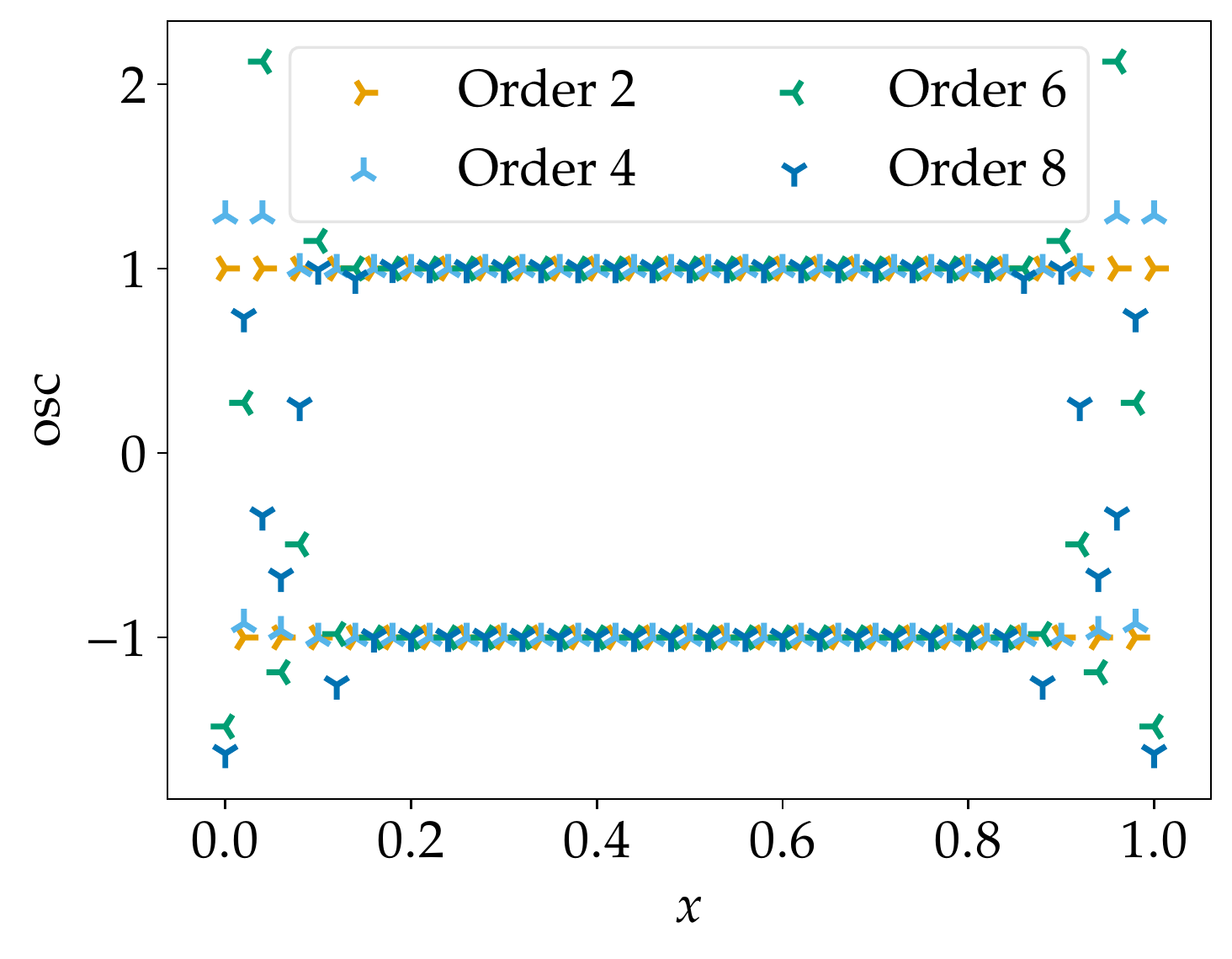}
    \caption{$N = 51$ grid points.}
  \end{subfigure}
  \caption{Grid oscillations for the SBP derivative operators of \cite{mattsson2004summation}
           and $N \in \set{50, 51}$ grid points.}
  \label{fig:grid_oscillations}
\end{figure}

\begin{example}
  For a nodal Lobatto Legendre (global) spectral method using polynomials of degree
  $\leq p$ and their exact derivatives, grid oscillations are given by the highest
  Legendre mode existing on the grid. Indeed, grid oscillations have to be polynomials
  of degree $\leq p$, orthogonal to all polynomials of degree $\leq p-1$.
\end{example}

\subsection{Two Space Dimensions}
\label{sec:SBP-potentials-2D}

In this section, the kernels of the discrete divergence and curl operators will be
characterised. It will become clear that scalar/vector potentials of discretely
curl/divergence free vector fields exist if and only if no grid oscillations are
present.

\begin{theorem}
\label{thm:SBP-scalar-potential-2D}
  Suppose that nullspace consistent tensor product SBP operators are applied
  in two space dimensions.
  Then,
  \begin{equation}
    \dim \image \grad = N_1 N_2 - 1 < N_1 N_2 + 1 = \dim \kernel \curl
  \end{equation}
  and the kernel of the discrete curl operator can be decomposed into the direct
  orthogonal sum
  \begin{equation}
  \label{eq:SBP-kernel-curl-2D}
    \kernel \curl
    =
    \image \grad \oplus
    \spann \set{
      \begin{pmatrix} \osc_1 \\ \vec{0} \end{pmatrix},
      \begin{pmatrix} \vec{0} \\ \osc_2 \end{pmatrix}
    }.
  \end{equation}
\end{theorem}
\begin{proof}
  Since the operator is nullspace consistent, $\kernel \grad = \spann \set{\vec{1}}$
  and
  \begin{equation}
    \dim \image \grad = N_1 N_2 - \dim \kernel \grad = N_1 N_2 - 1.
  \end{equation}
  Similarly,
  \begin{equation}
  \begin{aligned}
    \dim \kernel \curl
    &=
    \dim \kernel \begin{pmatrix} -D_2, D_1 \end{pmatrix}
    =
    \dim \left( \image \begin{pmatrix} -D_2^* \\ D_1^* \end{pmatrix} \right)^\perp
    \\
    &=
    2 N_1 N_2 - \dim \image \begin{pmatrix} -D_2^* \\ D_1^* \end{pmatrix}
    =
    N_1 N_2 + \dim \kernel \begin{pmatrix} -D_2^* \\ D_1^* \end{pmatrix}
    =
    N_1 N_2 + 1.
  \end{aligned}
  \end{equation}

  Since tensor product derivative operators commute, $\image \grad \subseteq \kernel \curl$.
  Additionally,
  \begin{equation}
    \image \grad
    =
    \image \begin{pmatrix} D_1 \\ D_2 \end{pmatrix}
    =
    \left( \kernel \begin{pmatrix} D_1^*, D_2^* \end{pmatrix} \right)^\perp
    =
    (\kernel \grad^*)^\perp
  \end{equation}
  and the span in \eqref{eq:SBP-kernel-curl-2D} is contained in $\kernel \grad^*$.
\end{proof}

\begin{theorem}
\label{thm:SBP-vector-potential-2D}
  Suppose that nullspace consistent tensor product SBP operators are applied
  in two space dimensions.
  Then,
  \begin{equation}
    \dim \image \rot = N_1 N_2 - 1 < N_1 N_2 + 1 = \dim \kernel \div
  \end{equation}
  and the kernel of the discrete divergence operator can be decomposed into the
  direct orthogonal sum
  \begin{equation}
    \kernel \div
    =
    \image \rot \oplus
    \spann \set{
      \begin{pmatrix} \vec{0} \\ \osc_1 \end{pmatrix},
      \begin{pmatrix} \osc_2 \\ \vec{0} \end{pmatrix}
    }.
  \end{equation}
\end{theorem}
\begin{proof}
  The arguments are basically the same as in the proof of
  Theorem~\ref{thm:SBP-scalar-potential-2D}, since
  \begin{equation}
    \rot
    =
    \begin{pmatrix} D_2 \\ -D_1 \end{pmatrix}
    =
    \begin{pmatrix} 0 & \I \\ -\I & 0 \end{pmatrix}
    \begin{pmatrix} D_1 \\ D_2 \end{pmatrix}
    =
    \begin{pmatrix} 0 & \I \\ -\I & 0 \end{pmatrix} \grad
  \end{equation}
  and
  \begin{equation}
    \div
    =
    \begin{pmatrix} D_1, D_2 \end{pmatrix}
    =
    \begin{pmatrix} -D_2, D_1 \end{pmatrix}
    \begin{pmatrix} 0 & -\I \\ \I & 0 \end{pmatrix}
    =
    \curl \begin{pmatrix} 0 & -\I \\ \I & 0 \end{pmatrix}.
    \qedhere
  \end{equation}
\end{proof}

\subsection{Scalar Potentials via Integrals in Two Space Dimensions}
\label{sec:SBP-scalar-potential-integral-2D}

Theorem~\ref{thm:SBP-scalar-potential-2D} shows that not every discretely curl free
vector field is the gradient of a scalar potential and specifies even the orthogonal
complement of $\image \grad$ in $\kernel \curl$ in two space dimensions.
In the continuous setting of classical vector calculus, scalar potentials are
often constructed explicitly using integrals. Hence, it is interesting to review
this construction and its discrete analogue, yielding another proof of
\eqref{eq:SBP-kernel-curl-2D}.

Classically, a scalar potential of a (sufficiently smooth) curl free vector field
$u$ in the box $[0, x_{1,\mathrm{max}}] \times [0, x_{2,\mathrm{max}}]$ can be
defined via
\begin{equation}
\label{eq:scalar-potential-integral-2D}
  \phi(x) = \int_0^{x_1} u_1(\xi, 0) \dif \xi + \int_0^{x_2} u_2(x_1, \eta) \dif \eta.
\end{equation}
Indeed, $\partial_2 \phi(x) = u_2(x)$ and
\begin{equation}
  \partial_1 \phi(x)
  =
  u_1(x_1, 0)
  + \int_0^{x_2} \underbrace{\partial_1 u_2}_{= \partial_2 u_1}(x_1, \eta) \dif \eta
  =
  u_1(x_1, x_2).
\end{equation}
% or
% \begin{equation}
%   \phi(x) = \int_0^{x_1} u_1(\xi, x_2) \dif \xi + \int_0^{x_2} u_2(0, \eta) \dif \eta.
% \end{equation}

Consider now a discretely curl free vector field $\vec{u}$ perpendicular to
$\spann \set{ \begin{pmatrix} \osc_1 \\ \vec{0} \end{pmatrix},
\begin{pmatrix} \vec{0} \\ \osc_2 \end{pmatrix} }$,
i.e. a discrete vector field $\vec{u}$ satisfying
\begin{equation}
\label{eq:SBP-scalar-potential-integral-2D-conditions}
  D_1 \vec{u}_2 = D_2 \vec{u}_1,
  \qquad
  \vec{u}_1 \perp \osc_1,
  \qquad
  \vec{u}_2 \perp \osc_2.
\end{equation}
Since the second integral in \eqref{eq:scalar-potential-integral-2D} is the inverse
of the partial derivative $\partial_2$, the discrete $\vec{u}_2$ must be in $\image D_2$
in order to mimic \eqref{eq:scalar-potential-integral-2D} discretely.
\begin{lemma}
\label{lem:SBP-scalar-potential-integral-2D-1}
  Suppose that nullspace consistent tensor product SBP operators are applied
  in two space dimensions.
  If the discrete vector field $\vec{u}$ satisfies
  \eqref{eq:SBP-scalar-potential-integral-2D-conditions},
  $\vec{u}_i \in \image D_i$, $i \in \set{1, 2}$.
\end{lemma}
\begin{proof}
  It suffices to consider the case $i = 2$ ($i = 1$ is similar).

  There are $\vec{v}_2, \vec{w}_2$ such that
  $\vec{u}_2 = D_2 \vec{v}_2 + \vec{w}_2$, where
  $\vec{w}_2 \in (\image D_2)^\perp = \kernel D_2^*$.
  To show that $\vec{w}_2 = \vec{0}$, use $D_2 \vec{u}_1 = D_1 \vec{u}_2 = D_1 D_2 \vec{v}_2 + D_1 \vec{w}_2$ and
  calculate
  \begin{multline}
    \norm{D_1 \vec{w}_2}^2
    =
    \scp{D_1 \vec{w}_2}{D_2 \vec{u}_1 - D_1 D_2 \vec{v}_2}
    =
    \scp{\vec{w}_2}{D_1^* D_2 \vec{u}_1 - D_1^* D_1 D_2 \vec{v}_2}
    \\
    =
    \scp{\vec{w}_2}{D_2 \bigl( D_1^* \vec{u}_1 - D_1^* D_1 \vec{v}_2 \bigr)}
    =
    0.
  \end{multline}
  Therefore, $\vec{w}_2 \in \kernel D_1 = \spann \set{\vec{1}} \otimes \R^{N_2}$. However,
  $\vec{w}_2 \in \kernel D_2^* = \R^{N_1} \otimes \spann \set{\osc_y}$ by definition.
  Hence, $\vec{w}_2 \in \spann \set{\vec{1} \otimes \osc_y} = \spann \set{\osc_2}$. Finally,
  using $\vec{u}_2 \perp \osc_2$ yields
  \begin{equation}
    0
    =
    \scp{\vec{u}_2}{\osc_2}
    =
    \scp{D_2 \vec{v}_2 + \vec{w}_2}{\osc_2}
    =
    \scp{\vec{v}_2}{D_2^* \osc_2} + \scp{\vec{w}_2}{\osc_2}
    =
    \scp{\vec{w}_2}{\osc_2},
  \end{equation}
  since $\osc_2 \in \kernel D_2^*$. Thus, $\vec{w}_2 = \vec{0}$.
\end{proof}

Next, discrete inverse operators of the partial derivatives are needed in order
to mimic the integrals in \eqref{eq:scalar-potential-integral-2D}. At first, the
one dimensional setting will be studied in the following.

Consider a nullspace consistent SBP derivative operator $D$ on the interval
$[0, x_{\mathrm{max}}]$ using $N$ grid points and the corresponding subspaces
\begin{equation}
  V_0 = \set{\vec{u} \in \R^N | \vec{u}(x = 0) = 0 },
  \quad
  V_1 = \set{\vec{u} \in \R^N | \exists \vec{v} \in \R^N \colon \vec{u} = D \vec{v}}.
\end{equation}
Here and in the following, $\vec{u}(x = 0)$ denotes the value of the discrete function
$\vec{u}$ at the corresponding grid points. In the one-dimensional case, $\vec{u}(x = 0) = \vec{u}^{(1)}$
is the first coefficient of $\vec{u}$. This notation is useful in several space dimensions
to refer to values at hyperplanes and other subspaces.

Clearly, $D\colon V_0 \to V_1$ is surjective. Because of nullspace consistency,
$D\colon V_0 \to V_1$ is even bijective and hence invertible. Denote the inverse
operator as $D^{-1}\colon V_1 \to V_0$.
In multiple space dimensions, the discrete partial derivative operators and their
inverse operators are defined analogously using tensor products.
Now, everything is set to provide another proof of \eqref{eq:SBP-kernel-curl-2D}.
\begin{lemma}
\label{lem:SBP-scalar-potential-integral-2D-2}
  Suppose that nullspace consistent tensor product SBP operators are applied
  in two space dimensions.
  If the discrete vector field $\vec{u}$ satisfies
  \eqref{eq:SBP-scalar-potential-integral-2D-conditions},
  there is a discrete scalar potential $\vec{\phi}$ of $\vec{u}$.
\end{lemma}
\begin{corollary}
  Suppose that nullspace consistent tensor product SBP operators are applied
  in two space dimensions.
  Then, $\dim \kernel \curl = \dim \image \grad + 2$.
\end{corollary}
\begin{proof}[Proof of Lemma~\ref{lem:SBP-scalar-potential-integral-2D-2}]
  By Lemma~\ref{lem:SBP-scalar-potential-integral-2D-1}, $\vec{u}_i \in \image D_i$
  and $D_i^{-1} \vec{u}_i$ is well-defined for $i \in \set{1, 2}$. Define
  \begin{equation}
    \vec{\phi} = (D_1^{-1} \vec{u}_1)(x_2 = 0) \otimes \vec{1} + D_2^{-1} \vec{u}_2.
  \end{equation}
  Here, $(D_1^{-1} \vec{u}_1)(x_2 = 0)$ denotes the value of $D_1^{-1} \vec{u}_1$ in the
  $x_2 = 0$ hyperplane.
  Then,
  \begin{equation}
    D_2 \vec{\phi}
    =
    (\I_x \otimes D_y) \vec{\phi}
    =
    \vec{0} + (\I_x \otimes D_y D_y^{-1}) \vec{u}_2
    =
    \vec{u}_2.
  \end{equation}
  Moreover, using $D_1 \vec{u}_2 = D_2 \vec{u}_1$,
  \begin{equation}
    D_1 D_2^{-1} \vec{u}_2
    =
    (D_x \otimes \I_y) (\I_x \otimes D_y^{-1}) \vec{u}_2
    =
    D_2^{-1} D_1 \vec{u}_2
    =
    D_2^{-1} D_2 \vec{u}_1.
  \end{equation}
  Since $D_2^{-1}$ is the inverse of $D_2$ for fields with zero initial values
  at $x_2 = 0$,
  \begin{equation}
    D_1 \vec{\phi}
    =
    (D_x \otimes \I_y) \bigl( D_x^{-1} \vec{u}_1(x_2 = 0) \bigr) \otimes \vec{1}
    + D_2^{-1} D_2 \vec{u}_1
    =
    \vec{u}_1(x_2 = 0) \otimes \vec{1} + D_2^{-1} D_2 \vec{u}_1
    =
    \vec{u}_1.
  \end{equation}
  Hence, $\vec{\phi}$ is a scalar potential of $\vec{u}$ and \eqref{eq:scalar-potential-integral-2D}
  is mimicked discretely.
\end{proof}

\subsection{Preliminary Results in Three Space Dimensions}
\label{sec:SBP-scalar-potential-integral-3D-preliminary}

Here, the kernels of the discrete divergence and curl operators will be studied
in three space dimensions. Since the arguments seem to be more complicated than
in the two-dimensional case because of the different structure of the curl operator,
preliminary results are obtained at first. They will be improved using the same
techniques presented in Section~\ref{sec:SBP-scalar-potential-integral-2D} afterwards.

\begin{lemma}
\label{lem:SBP-scalar-potential-3D}
  Suppose that nullspace consistent tensor product SBP operators are applied
  in three space dimensions.
  Then,
  \begin{equation}
    \dim \image \grad = N_1 N_2 N_3 - 1 < N_1 N_2 N_3 + 2 \leq \dim \kernel \curl
  \end{equation}
  and the kernel of the discrete curl operator is a superspace of the direct
  orthogonal sum
  \begin{equation}
  \label{eq:SBP-kernel-curl-3D-preliminary}
    \kernel \curl
    \supseteq
    \image \grad \oplus
    \spann \set{
      \begin{pmatrix} \osc_1 \\ \vec{0} \\ \vec{0} \end{pmatrix},
      \begin{pmatrix} \vec{0} \\ \osc_2 \\ \vec{0} \end{pmatrix},
      \begin{pmatrix} \vec{0} \\ \vec{0} \\\osc_3 \end{pmatrix}
    }.
  \end{equation}
\end{lemma}
\begin{proof}
  For a nullspace consistent operator, $\kernel \grad = \spann \set{\vec{1}}$ and
  \begin{equation}
    \dim \image \grad = N_1 N_2 N_3 - \dim \kernel \grad = N_1 N_2 N_3 - 1.
  \end{equation}

  Since tensor product derivative operators commute, $\image \grad \subseteq \kernel \curl$.
  Additionally,
  \begin{equation}
    \image \grad
    =
    \image \begin{pmatrix} D_1 \\ D_2 \\ D_3 \end{pmatrix}
    =
    \left( \kernel \begin{pmatrix} D_1^*, D_2^*, D_3^* \end{pmatrix} \right)^\perp
    =
    (\kernel \grad^*)^\perp
  \end{equation}
  and the span in \eqref{eq:SBP-kernel-curl-3D-preliminary} is contained in both
  $\kernel \grad^*$ and $\kernel \curl$.
\end{proof}

\begin{lemma}
\label{lem:SBP-vector-potential-3D}
  Suppose that nullspace consistent tensor product SBP operators are applied
  in three space dimensions.
  Then,
  \begin{equation}
    \dim \image \curl \leq 2 N_1 N_2 N_3 - 2 < 2 N_1 N_2 N_3 + 1 = \dim \kernel \div
  \end{equation}
  and the kernel of the discrete divergence operator is a superspace of the direct
  orthogonal sum
  \begin{equation}
  \label{eq:SBP-kernel-div-3D-preliminary}
    \kernel \div
    \supseteq
    \image \curl \oplus
    \spann \set{
      \begin{pmatrix} \osc_{23} \\ \vec{0} \\ \vec{0} \end{pmatrix},
      \begin{pmatrix} \vec{0} \\ \osc_{13} \\ \vec{0} \end{pmatrix},
      \begin{pmatrix} \vec{0} \\ \vec{0} \\\osc_{12} \end{pmatrix}
    }.
  \end{equation}
\end{lemma}
\begin{proof}
  For a nullspace consistent operator,
  \begin{multline}
    \dim \kernel \div
    =
    \dim \left( \image \begin{pmatrix} D_1^* \\ D_2^* \\ D_3^* \end{pmatrix} \right)^\perp
    =
    3 N_1 N_2 N_3 - \dim \image \begin{pmatrix} D_1^* \\ D_2^* \\ D_3^* \end{pmatrix}
    \\
    =
    2 N_1 N_2 N_3 + \dim \kernel \begin{pmatrix} D_1^* \\ D_2^* \\ D_3^* \end{pmatrix}
    =
    2 N_1 N_2 N_3 + 1.
  \end{multline}

  Since tensor product derivative operators commute, $\image \curl \subseteq \kernel \div$.
  Additionally, $\image \curl =  (\kernel \curl^*)^\perp$ and the span in
  \eqref{eq:SBP-kernel-div-3D-preliminary} is contained in both $\kernel \curl^*$
  and $\kernel \div$.
\end{proof}

\subsection{Scalar Potentials via Integrals in Three Space Dimensions}
\label{sec:SBP-scalar-potential-integral-3D}

The methods used in Section~\ref{sec:SBP-scalar-potential-integral-2D} to get
scalar potential for discretely curl free vector fields can also be applied in
three space dimensions. They can even be used to extend the preliminary results
of the previous Section~\ref{sec:SBP-scalar-potential-integral-3D-preliminary}.

Consider a box $[0, x_{1,\mathrm{max}}] \times [0, x_{2,\mathrm{max}}] \times
[0, x_{3,\mathrm{max}}]$. In the continuous setting, the analogue of
\eqref{eq:scalar-potential-integral-2D} is
\begin{equation}
\label{eq:scalar-potential-integral-3D}
  \phi(x)
  =
  \int_0^{x_1} u_1(\xi, 0, 0) \dif \xi
  + \int_0^{x_2} u_2(x_1, \eta, 0) \dif \eta
  + \int_0^{x_3} u_3(x_1, x_2, \zeta) \dif \zeta.
\end{equation}

As in the two-dimensional case, $\vec{u}_i \in \image D_i$ is necessary to mimic
\eqref{eq:scalar-potential-integral-3D} discretely. Here, the necessary conditions
are
\begin{equation}
\label{eq:SBP-scalar-potential-integral-3D-conditions}
  \forall i,j \in \set{1, 2, 3}\colon
  \qquad
  D_i \vec{u}_j = D_j \vec{u}_i,
  \qquad
  \vec{u}_i \perp \osc_i.
\end{equation}
\begin{lemma}
\label{lem:SBP-scalar-potential-integral-3D-1}
  Suppose that nullspace consistent tensor product SBP operators are applied
  in three space dimensions.
  If the discrete vector field $\vec{u}$ satisfies
  \eqref{eq:SBP-scalar-potential-integral-3D-conditions},
  $\vec{u}_i \in \image D_i$, $i \in \set{1, 2, 3}$.
\end{lemma}
\begin{proof}
  Consider $i = 1$ for simplicity. The other cases can be handled similarly.

  There are $\vec{v}_1, \vec{w}_1$ such that $\vec{u}_1 = D_1 \vec{v}_1 + \vec{w}_1$, where
  $\vec{w}_1 \in (\image D_1)^\perp = \kernel D_1^*$.
  To show that $\vec{w}_1 = \vec{0}$, use $D_1 \vec{u}_j = D_j \vec{u}_1 = D_j D_1 \vec{v}_1 + D_j \vec{w}_1$ for
  $j \in \set{2, 3}$ and calculate (without summing over $j$)
  \begin{multline}
    \norm{D_j \vec{w}_1}^2
    =
    \scp{D_j \vec{w}_1}{D_1 \vec{u}_j - D_j D_1 \vec{v}_1}
    =
    \scp{\vec{w}_1}{D_j^* D_1 \vec{u}_j - D_j^* D_j D_1 \vec{v}_1}
    \\
    =
    \scp{\vec{w}_1}{D_1 \bigl( D_j^* \vec{u}_j - D_j^* D_j \vec{v}_1 \bigr)}
    =
    0.
  \end{multline}
  Therefore, $\vec{w}_1 \in \kernel D_2 \cap \kernel D_3 = \R^{N_1} \otimes \spann \set{\vec{1}}
  \otimes \spann \set{\vec{1}}$. However, $\vec{w}_1 \in \kernel D_1^* = \spann \set{\osc_x}
  \otimes \R^{N_2} \otimes \R^{N_3}$ by definition.
  Hence, $\vec{w}_1 \in \spann \set{\osc_x \otimes \vec{1} \otimes \vec{1}} = \spann \set{\osc_1}$.
  Finally, using $\vec{u}_1 \perp \osc_1$ yields
  \begin{equation}
    0
    =
    \scp{\vec{u}_1}{\osc_1}
    =
    \scp{D_1 \vec{v}_1 + \vec{w}_1}{\osc_1}
    =
    \scp{\vec{v}_1}{D_1^* \osc_1} + \scp{\vec{w}_1}{\osc_1}
    =
    \scp{\vec{w}_1}{\osc_1},
  \end{equation}
  since $\osc_1 \in \kernel D_1^*$. Thus, $\vec{w}_1 = \vec{0}$.
\end{proof}

Using Lemma~\ref{lem:SBP-scalar-potential-integral-3D-1} allows to prove
\begin{lemma}
\label{lem:SBP-scalar-potential-integral-3D-2}
  Suppose that nullspace consistent tensor product SBP operators are applied
  in three space dimensions.
  If the discrete vector field $\vec{u}$ satisfies
  \eqref{eq:SBP-scalar-potential-integral-3D-conditions},
  there is a discrete scalar potential $\vec{\phi}$ of $\vec{u}$.
\end{lemma}
\begin{corollary}
\label{cor:SBP-scalar-potential-integral-3D}
  Suppose that nullspace consistent tensor product SBP operators are applied
  in three space dimensions.
  Then, $\dim \kernel \curl = \dim \image \grad + 3$.
\end{corollary}
\begin{proof}[Proof of Lemma~\ref{lem:SBP-scalar-potential-integral-3D-2}]
  By Lemma~\ref{lem:SBP-scalar-potential-integral-2D-1}, $\vec{u}_i \in \image D_i$
  and $D_i^{-1} \vec{u}_i$ is well-defined for $i \in \set{1, 2, 3}$. Define
  \begin{equation}
    \vec{\phi}
    =
    (D_1^{-1} \vec{u}_1)(x_2 = x_3 = 0) \otimes \vec{1} \otimes \vec{1}
    + (D_2^{-1} \vec{u}_2)(x_3 = 0) \otimes \vec{1}
    + D_3^{-1} \vec{u}_3.
  \end{equation}
  Here, $(D_1^{-1} \vec{u}_1)(x_2 = x_3 = 0)$ denotes the value of
  $D_1^{-1} \vec{u}_1$ on the $x_2 = x_3 = 0$ curve and
  $(D_2^{-1} \vec{u}_2)(x_3 = 0)$ is a value in the $x_3 = 0$ plane.
  Then,
  \begin{equation}
    D_3 \vec{\phi}
    =
    (\I_x \otimes \I_y \otimes D_z) \vec{\phi}
    =
    \vec{0} + \vec{0} + (\I_x \otimes \I_y \otimes D_z D_z^{-1}) \vec{u}_3
    =
    \vec{u}_3.
  \end{equation}
  Using $D_2 \vec{u}_3 = D_3 \vec{u}_2$ yields
  \begin{equation}
    D_2 D_3^{-1} \vec{u}_3
    =
    D_3^{-1} D_2 \vec{u}_3
    =
    D_3^{-1} D_3 \vec{u}_2.
  \end{equation}
  Since $D_3^{-1}$ is the inverse of $D_3$ for fields with zero initial values
  at $x_3 = 0$,
  \begin{multline}
    D_2 \vec{\phi}
    =
    \vec{0}
    + (\I_x \otimes D_y \otimes \I_z) \bigl( (\I_x \otimes D_y^{-1}) \vec{u}_2(x_3 = 0) \bigr) \otimes \vec{1}
    + D_3^{-1} D_3 \vec{u}_2
    \\
    =
    \vec{u}_2(x_3 = 0) \otimes \vec{1} + D_3^{-1} D_3 \vec{u}_2
    =
    \vec{u}_2.
  \end{multline}
  Similarly, inserting $D_1 \vec{u}_2 = D_2 \vec{u}_1$ results in
  \begin{equation}
    D_1 D_2^{-1} \vec{u}_2
    =
    D_2^{-1} D_1 \vec{u}_2
    =
    D_2^{-1} D_2 \vec{u}_1
  \end{equation}
  and $D_1 \vec{u}_3 = D_3 \vec{u}_1$ yields
  \begin{equation}
    D_1 D_3^{-1} \vec{u}_3
    =
    D_3^{-1} D_1 \vec{u}_3
    =
    D_3^{-1} D_3 \vec{u}_1.
  \end{equation}
  Using again that $D_i^{-1}$, $i \in \set{2,3}$, is the inverse of $D_i$ for
  fields with zero initial values at $x_i = 0$,
  \begin{multline}
    D_1 \vec{\phi}
    =
    (D_x \otimes \I_y \otimes \I_z) \bigl( D_x^{-1} \vec{u}_1(x_2 = x_3 = 0) \bigr) \otimes \vec{1} \otimes \vec{1}
    + (D_2^{-1} D_2 \vec{u}_1) (x_3 = 0) \otimes \vec{1}
    + D_3^{-1} D_3 \vec{u}_1
    \\
    =
      \vec{u}_1(x_2 = x_3 = 0) \otimes \vec{1} \otimes \vec{1}
    + \vec{u}_1(x_3 = 0) \otimes \vec{1} - \vec{u}_1(x_2 = x_3 = 0) \otimes \vec{1} \otimes \vec{1}
    + D_3^{-1} D_3 \vec{u}_1
    =
    \vec{u}_1.
  \end{multline}
  Hence, $\vec{\phi}$ is a scalar potential of $\vec{u}$ and
  \eqref{eq:scalar-potential-integral-3D} is mimicked discretely.
\end{proof}

\begin{remark}
  Similarly to the construction of the scalar potential $\phi$
  \eqref{eq:scalar-potential-integral-3D}, a vector potential $v = (v_1, v_2, 0)$
  of a (sufficiently smooth) divergence free vector field $u$ can be constructed
  via
  \begin{equation}
    v_1(x)
    =
    \int_0^{x_3} u_2(x_1, x_2, \zeta) \dif \zeta
    - \int_0^{x_2} u_3(x_1, \eta, 0) \dif \eta,
    \qquad
    v_2(x)
    =
    -\int_0^{x_3} u_1(x_1, x_2, \zeta) \dif \zeta.
  \end{equation}
  Discrete versions can probably be obtained along the same lines.
\end{remark}

\subsection{Three Space Dimensions Revisited}
\label{sec:SBP-potentials-3D}

Using the results of the previous Section~\ref{sec:SBP-scalar-potential-integral-3D},
the following analogues in three space dimensions of
Theorems~\ref{thm:SBP-scalar-potential-2D} and \ref{thm:SBP-vector-potential-2D}
can be obtained.
In particular, the inequalities in Lemmas~\ref{lem:SBP-scalar-potential-3D}
and \ref{lem:SBP-vector-potential-3D} become equalities and scalar/vector potentials
of discretely curl/divergence free vector fields exist if and only if no grid
oscillations are present.

\begin{theorem}
\label{thm:SBP-scalar-potential-3D}
  Suppose that nullspace consistent tensor product SBP operators are applied
  in three space dimensions.
  Then,
  \begin{equation}
    \dim \image \grad = N_1 N_2 N_3 - 1 < N_1 N_2 N_3 + 2 = \dim \kernel \curl
  \end{equation}
  and the kernel of the discrete curl operator can be decomposed into the direct
  orthogonal sum
  \begin{equation}
  \label{eq:SBP-kernel-curl-3D}
    \kernel \curl
    =
    \image \grad \oplus
    \spann \set{
      \begin{pmatrix} \osc_1 \\ \vec{0} \\ \vec{0} \end{pmatrix},
      \begin{pmatrix} \vec{0} \\ \osc_2 \\ \vec{0} \end{pmatrix},
      \begin{pmatrix} \vec{0} \\ \vec{0} \\\osc_3 \end{pmatrix}
    }.
  \end{equation}
\end{theorem}
\begin{proof}
  Apply Lemma~\ref{lem:SBP-scalar-potential-3D} and
  Corollary~\ref{cor:SBP-scalar-potential-integral-3D}.
\end{proof}

\begin{theorem}
\label{thm:SBP-vector-potential-3D}
  Suppose that nullspace consistent tensor product SBP operators are applied
  in three space dimensions.
  Then,
  \begin{equation}
    \dim \image \curl = 2 N_1 N_2 N_3 - 2 < 2 N_1 N_2 N_3 + 1 = \dim \kernel \div
  \end{equation}
  and the kernel of the discrete divergence operator can be decomposed into the direct
  orthogonal sum
  \begin{equation}
  \label{eq:SBP-kernel-div-3D}
    \kernel \div
    =
    \image \curl \oplus
    \spann \set{
      \begin{pmatrix} \osc_{23} \\ \vec{0} \\ \vec{0} \end{pmatrix},
      \begin{pmatrix} \vec{0} \\ \osc_{13} \\ \vec{0} \end{pmatrix},
      \begin{pmatrix} \vec{0} \\ \vec{0} \\\osc_{12} \end{pmatrix}
    }.
  \end{equation}
\end{theorem}
\begin{proof}
  Apply Lemma~\ref{lem:SBP-vector-potential-3D} and
  Corollary~\ref{cor:SBP-scalar-potential-integral-3D}, using that
  $\dim \image \curl + \dim \kernel \curl = 3  N_1 N_2 N_3$.
\end{proof}

\subsection{Remarks on Numerical Implementations}
\label{sec:SBP-potentials-implementation}

Theorems~\ref{thm:SBP-scalar-potential-2D}, \ref{thm:SBP-vector-potential-2D},
\ref{thm:SBP-scalar-potential-3D}, and \ref{thm:SBP-vector-potential-3D}
show that $\kernel \curl = \image \grad$ and $\kernel \div = \image \curl$ (or
$\kernel \div = \image \rot$ in two space dimensions) do not
hold discretely. However, these relations become true when the kernels are restricted
to the subspace of grid functions orthogonal to grid oscillations in either
coordinate direction.

Hence, if potentials of curl/divergence free vector fields are sought, one has
to remove these grid oscillations, e.g. by an orthogonal projection. Such a
projection can be interpreted as a discrete filtering process, reducing the discrete
norm induced by the mass matrix. For example, the operator filtering out all
grid oscillations $\osc_i$, $i \in \set{1,\dots,d}$, is given by
\begin{equation}
\label{eq:filter-operator}
  F := \I - \sum_{i=1}^d \frac{\osc_i \osc_i^T M}{\norm{\osc_i}_M^2}.
\end{equation}
\begin{theorem}
  The filter operator $F$ \eqref{eq:filter-operator} is an orthogonal projection
  with respect to the scalar product induced by $M$
  and satisfies $\norm{F} \leq 1$, where the operator norm is induced by the
  discrete norm $\norm{\cdot}_M$.
\end{theorem}
\begin{proof}
  It suffices to note that grid oscillations in different coordinate directions
  are orthogonal, because
  \begin{equation}
    \vec{1} \in \image D_i \perp \kernel D_i^* = \spann \set{\osc_i}.
    \qedhere
  \end{equation}
\end{proof}

The approaches to construct scalar (and similarly vector) potentials in Sections
\ref{sec:SBP-potentials-2D} and \ref{sec:SBP-potentials-3D} depend crucially on
the satisfaction of $\curl \vec{u} = \vec{0}$ (or $\div \vec{u} = \vec{0}$) discretely. Hence, they can
be ill-conditioned and numerical roundoff errors can influence the results, cf.
\cite{silberman2019numerical} for a related argument concerning a ``direct'' and
a ``global linear algebra'' approach to compute vector potentials.
Moreover, they are not really suited for discrete Helmholtz Hodge decompositions
targeted in Section~\ref{sec:helmholtz}. Hence, other approaches will be pursued
in the following, cf. Section~\ref{sec:numerical-implementation}.

\section{Characterisation of Divergence and Curl Free Functions}
\label{sec:div-curl-free}

Combining Theorems~\ref{thm:SBP-scalar-potential-2D}, \ref{thm:SBP-vector-potential-2D},
\ref{thm:SBP-scalar-potential-3D}, and \ref{thm:SBP-vector-potential-3D} yields
a characterisation of vector fields that are both divergence free and curl free.
As before, the continuous case is described at first, cf. \cite[Corollary~2,
Theorem~2 and its proof]{schweizer2018friedrichs}.

\begin{theorem}
\label{thm:continuous-div-curl-0}
  For a vector field $u \in L^2(\Omega)^d$, the following conditions are equivalent.
  \begin{enumerate}[label=\roman*)]
    \item
    $\div u = 0$ and $\curl u = 0$.

    \item
    $u$ is the gradient of a harmonic function $\phi \in H^1(\Omega)$, solving
    the Neumann problem
    \begin{equation}
    \label{eq:div-curl-0-scalar-potential}
      \int_\Omega (\grad \phi) \cdot (\grad \psi)
      =
      \int_{\partial\Omega} (u \cdot \nu) \psi,
      \quad \forall \psi \in H^1(\Omega).
    \end{equation}
%     \item
%     In three space dimensions, $u$ is the curl of a harmonic vector field
%     $v \in H(\curl; \Omega) \cap H(\div; \Omega)$, $\div v = 0$, solving the problem
%     \begin{equation}
%     \label{eq:div-curl-0-vector-potential}
%       \int_\Omega (\curl v) \cdot (\curl w)
%       =
%       \int_{\partial\Omega} u \cdot (\nu \times w),
%       \quad \forall w \in H(\curl; \Omega).
%     \end{equation}
%     In two space dimensions, $\curl$ has to be substituted by $\rot$.
  \end{enumerate}
  Here, $\nu$ is the outer unit normal at $\partial \Omega$.
\end{theorem}
Note that the right hand side of \eqref{eq:div-curl-0-scalar-potential} is
well defined, because the trace of
$\psi \in H^1(\Omega)$
is in $H^{1/2}(\partial \Omega)$ and the normal trace of
$u \in H(\div; \Omega) = \set{ u \in L^2(\Omega) | \div u \in L^2(\Omega)}$
is in $H^{-1/2}(\partial \Omega)$ \cite[Theorem~I.2.5]{girault1986finite}.

Although the theorems guaranteeing the existence of scalar/vector potentials do
not hold discretely, the characteristation of vector fields that are both divergence
and curl free is similar to the one at the continuous level given by
Theorem~\ref{thm:continuous-div-curl-0}.
\begin{theorem}
\label{thm:SBP-div-curl-0}
  If nullspace consistent tensor product SBP operators are applied, the following
  conditions are equivalent for a grid function $\vec{u}$ in two or three space dimensions.
  \begin{enumerate}[label=\roman*)]
    \item \label{itm:SBP-div-curl-0}
    $\div \vec{u} = \vec{0}$ and $\curl \vec{u} = \vec{0}$.

    \item \label{itm:SBP-div-curl-0-scalar-potential}
    $\vec{u}$ is the discrete gradient of a discretely harmonic function
    $\vec{\phi}$, i.e. $D_i D_i \vec{\phi} = \vec{0}$.
  \end{enumerate}
  If $\div \vec{u} = \vec{0}$ and $\curl \vec{u} = \vec{0}$, the scalar potential
  $\vec{\phi}$ can be determined as solution of the Neumann problem
  \begin{equation}
  \label{eq:SBP-div-curl-0-scalar-potential}
%     \int_\Omega (\grad \phi) \cdot (\grad \psi)
    D_i^T M D_i \vec{\phi}
    =
%     \int_{\partial\Omega} (u \cdot \nu) \psi, \quad \forall \psi.
    E_i \vec{u}_i.
  \end{equation}
\end{theorem}
\begin{proof}
  Since the grid oscillations appearing in $\kernel \div$ are not in $\kernel \curl$
  and vice versa,
  \begin{equation}
    \kernel \div \cap \kernel \curl \subseteq \image \grad \cap \image \rot
  \end{equation}
  in two space dimensions and
  \begin{equation}
    \kernel \div \cap \kernel \curl \subseteq \image \grad \cap \image \curl
  \end{equation}
  in three space dimensions.

  ``\ref{itm:SBP-div-curl-0} $\iff$ \ref{itm:SBP-div-curl-0-scalar-potential}'':
  Because of $\vec{u} \in \image \grad$ and $\div \vec{u} = \vec{0}$, there is a scalar potential
  $\vec{\phi}$ satisfying $\vec{u}_i = D_i \vec{\phi}$ and $\vec{0} = D_i \vec{u}_i = D_i D_i \vec{\phi}$. Additionally,
  \begin{equation}
    D_i^T M D_i \vec{\phi} = E_i D_i \vec{\phi} - M D_i D_i \vec{\phi} = E_i \vec{u}_i.
  \end{equation}
  Conversely, if $\vec{\phi}$ is a discretely harmonic grid function and $\vec{u}_i = D_i \vec{\phi}$,
  $\div \vec{u} = D_i \vec{u}_i = D_i D_i \vec{\phi} = \vec{0}$ because
  $\vec{\phi}$ is discretely harmonic.
  Additionally, $\curl \vec{u} = \vec{0}$, since $\vec{u}$ is the discrete
  gradient of $\vec{\phi}$ and the discrete curl of a discrete gradient
  vanishes, cf. Remark~\ref{rem:discrete-derivatives-commute}.

  A solution of the discrete Neumann problem \eqref{eq:SBP-div-curl-0-scalar-potential}
  is determined uniquely up to an additive constant, since
  $\kernel D_i^T M D_i = \kernel \grad = \set{\vec{1}}$
  for nullspace consistent SBP operators. Hence, $D_i^T M D_i$ is symmetric and
  positive semidefinite and a solution of the Neumann problem exists if and only
  if the right hand side is orthogonal to the kernel of $D_i^T M D_i$ (both with
  respect to the Euclidean standard inner product and not the one induced by $M$).
  This is the case if $D_i \vec{u}_i = \vec{0}$, since
  \begin{equation}
    \vec{1}^T E_i \vec{u}_i
    =
    \vec{1}^T (M D_i + D_i^T M) \vec{u}_i
    =
    \vec{1}^T M D_i \vec{u}_i.
    \qedhere
  \end{equation}
\end{proof}

\section{Variants of the Helmholtz Hodge Decomposition}
\label{sec:helmholtz}

There are several variants of the Helmholtz Hodge decomposition of a vector field
$u \in L^2(\Omega)$, i.e. decompositions of $u$ into curl free and divergence free
components, e.g.
\begin{equation}
  \forall u \in L^2(\Omega)^2 \;\exists \phi \in H^1(\Omega), v \in H^1(\Omega)\colon
  \quad
  u = \grad \phi + \rot v
\end{equation}
in two space dimensions \cite[Theorem~I.3.2]{girault1986finite} and
\begin{equation}
  \forall u \in L^2(\Omega)^3 \;\exists \phi \in H^1(\Omega), v \in H(\curl; \Omega)\colon
  \quad
  u = \grad \phi + \curl v
\end{equation}
in three space dimensions \cite[Corollary~I.3.4]{girault1986finite}, where additional
(boundary) conditions are used to specify the potentials (e.g. uniquely up to an
additive constant for the scalar potential $\phi$) and guarantee that these
decompositions are orthogonal in $L^2(\Omega)$.
Discretely, such decompositions are not possible in general.
\begin{theorem}
\label{thm:SBP-helmholtz-impossible}
  For nullspace consistent tensor product SBP operators, there are grid functions
  $\vec{u}$ such that
  \begin{equation}
    \begin{cases}
      \vec{u} \notin \image \grad + \image \rot, & \text{ in two space dimensions}, \\
      \vec{u} \notin \image \grad + \image \curl, & \text{ in three space dimensions}.
    \end{cases}
  \end{equation}
  In particular,
  \begin{equation}
    \spann \set{
      \begin{pmatrix} \osc_{12} \\ \vec{0} \end{pmatrix},
      \begin{pmatrix} \vec{0} \\ \osc_{12} \end{pmatrix}
    }
    \subseteq (\image \grad + \image \rot)^\perp
  \end{equation}
  in two space dimensions and
  \begin{equation}
    \spann \set{
      \begin{pmatrix} \osc_{123} \\ \vec{0} \\ \vec{0} \end{pmatrix},
      \begin{pmatrix} \vec{0} \\ \osc_{123} \\ \vec{0} \end{pmatrix},
      \begin{pmatrix} \vec{0} \\ \vec{0} \\ \osc_{123} \end{pmatrix}
    }
    \subseteq (\image \grad + \image \curl)^\perp
  \end{equation}
  in three space dimensions.
\end{theorem}
\begin{proof}
  Using Theorems~\ref{thm:SBP-scalar-potential-2D} and \ref{thm:SBP-vector-potential-2D},
  \begin{equation}
    \dim (\image \grad + \image \rot)
    \leq
    \dim \image \grad + \dim \image \rot
    =
    2 N_1 N_2 - 2
  \end{equation}
  in two space dimensions. In three space dimensions, Theorems~
  \ref{thm:SBP-scalar-potential-3D} and \ref{thm:SBP-vector-potential-3D} yield
  \begin{equation}
    \dim (\image \grad + \image \curl)
    \leq
    \dim \image \grad + \dim \image \curl
    =
    3 N_1 N_2 N_3 - 3.
  \end{equation}

  Finally, note that grid oscillations are orthogonal to the image of SBP
  derivative operators.
\end{proof}

\begin{remark}
\label{rem:SBP-helmholtz-impossible}
  In general, there is no equality in the subspace relations of
  Theorem~\ref{thm:SBP-helmholtz-impossible}. Up to now, no complete characterisation
  of $(\image \grad + \image \curl)^\perp$ or $(\image \grad)^\perp \cap
  (\image \curl)^\perp = (\kernel \grad^*) \cap (\kernel \curl^*)$ has been obtained.
  In numerical experiments, some sort of grid oscillations always seem to be involved.
\end{remark}

Nevertheless, it is possible to compute orthogonal decompositions of the form
\begin{equation}
\label{eq:general-helmholtz-decomposition}
  \begin{cases}
    \vec{u} = \grad \vec{\phi} + \rot  \vec{v} + \vec{r}, & \vec{r} \perp \image \grad, \image \rot,
      \quad \text{ \phantom{c}in two space dimensions}, \\
    \vec{u} = \grad \vec{\phi} + \curl \vec{v} + \vec{r}, & \vec{r} \perp \image \grad, \image \curl,
      \quad \text{ in three space dimensions}.
  \end{cases}
\end{equation}
In the following, only the three dimensional case will be described. In two
space dimensions, some occurrences of $\curl$ have to be substituted by $\rot$.

In the literature, variants of the Helmholtz Hodge decomposition in bounded
domains are most often presented with an emphasis on boundary conditions, cf.
\cite{girault1986finite,fernandes1997magnetostatic,amrouche1998vector,schweizer2018friedrichs}.
With this emphasis, the following two choices of boundary conditions appear
most often in the literature.
\begin{proposition}
\label{pro:continuous-HHD-BCs}
  Suppose $u \in L^2(\Omega)^3$.
  \begin{enumerate}[label=\roman*)]
    \item
    There exist $\phi \in H^1(\Omega)$ with $\int_\Omega \phi = 0$ and
    $v \in H(\Omega,\curl) \cap H(\Omega,\div)$ with
    $\div v = 0$ and $\nu \times v|_{\partial \Omega} = 0$
    such that \eqref{eq:general-helmholtz-decomposition} with $r = 0$ is
    an orthogonal decomposition.

    \item
    There exist $\phi \in H^1_0(\Omega)$ and
    $v \in H(\Omega,\curl) \cap H(\Omega,\div)$ with
    $\div v = 0$ and $\nu \cdot v|_{\partial \Omega} = 0$
    such that \eqref{eq:general-helmholtz-decomposition} with $r = 0$ is
    an orthogonal decomposition.
  \end{enumerate}
\end{proposition}
In practice, these different variants of the Helmholtz Hodge decomposition
can be obtained by projecting $u \in L^2(\Omega)^3$ onto $\image \grad$
and $\image \curl$.
The projections onto the closed subspaces $\image \grad, \image \curl$
(with proper choice of domain of definition) of $L^2(\Omega)^3$ commute
if and only if the subspaces are orthogonal, which is not the case.
Hence, the order of the projections matters and there are (at least)
two different choices:
\begin{enumerate}
  \item
  Firstly, project $\vec{u}$ onto $\image \grad$, yielding $\vec{u} - \grad \vec{\phi} \perp \image \grad$.
  Secondly, project the remainder $\vec{u} - \grad \vec{\phi}$ onto $\image \curl$, yielding
  $\vec{r} = \vec{u} - \grad \vec{\phi} - \curl \vec{v} \perp \image \curl, \image \grad$.

  \item
  Firstly, project $\vec{u}$ onto $\image \curl$, yielding $\vec{u} - \curl \vec{v} \perp \image \curl$.
  Secondly, project the remainder $\vec{u} - \curl \vec{v}$ onto $\image \grad$, yielding
  $\vec{r} = \vec{u} - \curl \vec{v} - \grad \vec{\phi} \perp \image \grad, \image \curl$.
\end{enumerate}
At the continuous level, the interpretation of different variants of the
Helmholtz Hodge decomposition via projections is given as follows.

\begin{proposition}
\label{pro:continuous-projections-2-components}
  Suppose $u \in L^2(\Omega)^3$ and consider the derivative operators
  $\grad\colon H^1(\Omega) \to L^2(\Omega)^3$ and
  $\curl\colon H(\curl, \Omega) \to L^2(\Omega)^3$.
  \begin{enumerate}[label=\roman*)]
    \item \label{itm:continuous-grad-curl}
    Projecting $u$ onto $\image \grad$ and the remainder onto $\image \curl$
    yields an orthogonal decomposition \eqref{eq:general-helmholtz-decomposition}
    with $r = 0$, $\int_\Omega \phi = 0$, and
    $\div v = 0, \nu \times v|_{\partial \Omega} = 0$.

    \item \label{itm:continuous-curl-grad}
    Projecting $u$ onto $\image \curl$ and the remainder onto $\image \grad$
    yields an orthogonal decomposition \eqref{eq:general-helmholtz-decomposition}
    with $r = 0$, $\div v = 0, \nu \cdot v|_{\partial \Omega} = 0$, and
    $\phi|_{\partial \Omega} = 0$.
  \end{enumerate}
\end{proposition}
\begin{proof}[Sketch of the proof]
  The projection of $\tilde u \in L^2(\Omega)^3$ onto $\image \grad$ is given by
  the solution of the associated normal problem, i.e. the Neumann problem
  \begin{equation}
    \int_\Omega (\grad \psi) \cdot (\grad \phi) = \int_\Omega (\grad \psi) \tilde u,
    \quad
    \forall \psi \in H^1(\Omega),
  \end{equation}
  yielding a unique solution $\phi \in \nicefrac{H^1(\Omega)}{\kernel \grad}$.
  Since $\kernel \grad = \spann \set{1}$, $\nicefrac{H^1(\Omega)}{\kernel \grad}
  = \nicefrac{H^1(\Omega)}{\R}$ can be identified with $(\kernel \grad)^\perp
  = \set{ \phi \in H^1(\Omega) | \int_\Omega \phi = 0}$.

  Similarly, the projection of $\tilde u \in L^2(\Omega)$ onto $\image \curl$ is
  given by the solution of the associated normal problem, i.e.
  \begin{equation}
    \int_\Omega (\curl w) \cdot (\curl v) = \int_\Omega (\curl w) \tilde u,
    \quad
    \forall w \in H(\curl, \Omega),
  \end{equation}
  yielding a unique solution $v \in \nicefrac{H(\curl, \Omega)}{\kernel \curl}$.
%   Since $\kernel \curl = \image \grad$, $\nicefrac{H^1(\Omega)}{\kernel \curl}$
%   can be identified with $(\kernel \curl)^\perp = \set{v \in H(\curl,\Omega) |
%   \div v = 0, \nu \cdot v|_{\partial \Omega} = 0}$ (by addition of suitable
%   gradients).

  For both cases, $L^2(\Omega)^3 = \image \grad + \image \curl$ can be used to
  conclude $r = 0$.

  For \ref{itm:continuous-grad-curl}, $\phi$ is specified as required and the
  boundary condition $\nu \times v|_{\partial \Omega} = 0$ is implied by
  $\curl v \perp \image \grad$, since
  \begin{equation}
    \int_\Omega (\curl v) \cdot (\grad \psi)
    =
    \int_{\partial \Omega} (\nu \times v) \cdot (\grad \psi).
  \end{equation}
  The additional condition $\div v = 0$ can be obtained by adding a suitable
  gradient $\in \kernel \curl$ to $v$.

  For \ref{itm:continuous-curl-grad}, the conditions $\div v = 0,
  \nu \cdot v|_{\partial \Omega} = 0$ can be obtained by adding a suitable
  gradient $\in \kernel \curl$, solving an inhomogeneous Neumann problem.
  The boundary condition for $\phi$ is implied by the orthogonality condition
  $\grad \phi \perp \image \curl$, since
  \begin{equation}
    \int_\Omega (\grad \phi) \cdot (\curl w)
    =
    \int_{\partial \Omega} \phi \nu \cdot (\curl w).
    \qedhere
  \end{equation}
\end{proof}
\begin{remark}
  The Helmholtz Hodge decompositions of
  Proposition~\ref{pro:continuous-projections-2-components}
  are exactly the ones of \cite[Theorem~2]{schweizer2018friedrichs}, although
  the (existence) proof given there follows partially another order and does not
  mention the projection onto subspaces.
\end{remark}

The constraints on $\phi$ and $v$ given in
Proposition~\ref{pro:continuous-projections-2-components} cannot be mimicked
completely at the discrete level. While it is always possible to choose a discrete
scalar potential $\vec{\phi}$ with vanishing mean value (by adding a suitable constant),
prescription of boundary conditions and the divergence of $\vec{v}$ are not always
possible. For example, the Laplacian of a scalar field can be prescribed in
$\Omega$ at the continuous level and (Neumann, Dirichlet) boundary conditions
can be prescribed additionally. This is not possible at the discrete level in
all cases, since the system is overdetermined if both the derivative and boundary
values are prescribed at $\partial \Omega$, cf. \cite{ranocha2018numerical}.
Additionally, there is
\begin{theorem}
\label{thm:SBP-divergence-free-vector-potential}
  Suppose that nullspace consistent tensor product SBP operators which are at
  least first order accurate in the complete domain are applied in two or three
  space dimensions.
  Then,
  \begin{equation}
    \dim \image \div > \dim \image \div|_{\kernel \curl}.
  \end{equation}
  Hence, it is not always possible to choose a divergence free vector potential.
\end{theorem}
\begin{proof}
  Consider at first the case of three space dimensions.
  Using Theorem~\ref{thm:SBP-vector-potential-3D},
  \begin{equation}
    \dim \image \div
    =
    3 N_1 N_2 N_3 - \underbrace{\dim \kernel \div}_{= 2 N_1 N_2 N_3 + 1}
    =
    N_1 N_2 N_3 - 1.
  \end{equation}
  Using Theorem~\ref{thm:SBP-scalar-potential-3D},
  \begin{equation}
    \kernel \curl
    =
    \image \grad \oplus
    \spann \set{
      \begin{pmatrix} \osc_1 \\ \vec{0} \\ \vec{0} \end{pmatrix},
      \begin{pmatrix} \vec{0} \\ \osc_2 \\ \vec{0} \end{pmatrix},
      \begin{pmatrix} \vec{0} \\ \vec{0} \\\osc_3 \end{pmatrix}
    }.
  \end{equation}
  Hence,
  \begin{multline}
    \dim \image \div|_{\kernel \curl}
    \leq
    \dim \image \div|_{\image \grad} + 3
    \\
    =
    \dim \image D_i D_i + 3
    =
    N_1 N_2 N_3 - \dim \kernel D_i D_i + 3,
  \end{multline}
  where $D_i D_i$ is the discrete (wide stencil) Laplacian defined for scalar fields.
  Because of the accuracy of the SBP derivative operator,
  \begin{equation}
    \spann \set{\vec{1},
      \vec{x} \otimes \vec{1} \otimes \vec{1},
      \vec{1} \otimes \vec{y} \otimes \vec{1},
      \vec{1} \otimes \vec{1} \otimes \vec{z},
      \vec{x} \otimes \vec{y} \otimes \vec{1},
      \vec{x} \otimes \vec{1} \otimes \vec{z},
      \vec{1} \otimes \vec{y} \otimes \vec{z},
      \vec{x} \otimes \vec{y} \otimes \vec{z}
    }
  \end{equation}
  is a subspace of $\kernel D_i D_i$ and $\dim \kernel D_i D_i \geq 8$. Hence,
  \begin{equation}
    \dim \image \div
    =
    N_1 N_2 N_3 - 1
    >
    N_1 N_2 N_3 - 5
    \geq
    \dim \image \div|_{\kernel \curl}.
  \end{equation}

  In two space dimensions, the computations are similar and yield
  \begin{equation}
    \dim \image \div
    =
    N_1 N_2 - 1
    >
    N_1 N_2 - 2
    \geq
    \dim \image \div|_{\kernel \curl}.
    \qedhere
  \end{equation}
\end{proof}

\subsection{Numerical Implementation}
\label{sec:numerical-implementation}

In order to compute discrete Helmholtz Hodge decompositions, the projections onto
$\image \grad$, $\image \curl$ are performed numerically. In particular, least norm
least squares solutions will be sought, i.e.
\begin{equation}
  \min_{\vec{\phi}} \norm{\vec{\phi}}_M^2 \quad \st \vec{\phi} \in \argmin \norm{\vec{u} - \grad \vec{\phi}}_M^2
\end{equation}
and
\begin{equation}
  \min_{\vec{v}} \norm{\vec{v}}_M^2 \quad \st \vec{v} \in \argmin \norm{\vec{u} - \curl \vec{v}}_M^2.
\end{equation}
The same approach is used for scalar potentials of curl free vector fields
and divergence free vector fields in three space dimensions (substitute $\curl$
by $\rot$ in two space dimensions).

There are several iterative numerical methods to solve these problems such as
LSQR \cite{paige1982lsqr,paige1982algorithm} based on CG, LSMR \cite{fong2011lsmr}
based on MINRES, and LSLQ \cite{estrin2019lslq} based on SYMMLQ.
In order to use existing implementations of these methods which are based on
the Euclidean scalar product, a scaling will be described and applied in the
following. This scaling by the square root of the mass matrix transforms properties
of the iterative methods based on the Euclidean scalar product and norm (such as
error/residual monotonicity) to the norm induced by the mass matrix. Additionally,
the projections become orthogonal with respect to the SBP scalar product.
In three space dimensions, the scalings are
\begin{itemize}
  \item
  \verb|phi = sqrtM \ linsolve(sqrtMvec*grad/sqrtM, sqrtMvec*u)| for scalar
  potentials and

  \item
  \verb|v = sqrtMvec \ linsolve(sqrtMvec*curl/sqrtMvec, sqrtMvec*u)| for vector potentials,
\end{itemize}
where $\texttt{sqrtM} = \sqrt{M}$, $\texttt{sqrtMvec} = \I_3 \otimes \sqrt{M}$,
\verb|linsolve| denotes a linear solver such as LSQR or LSMR, and the other
notation should be clear. Note that the computation of the square root of the
mass matrix is inexpensive for diagonal mass matrices.

\section{Numerical Examples}
\label{sec:numerical-examples}

In this section, some numerical examples using the methods discussed hitherto
will be presented. The classical SBP operators of \cite{mattsson2004summation}
will be used, since they are widespread in applications. Optimised operators
such as the ones of \cite{mattsson2014optimal,mattsson2018boundary} would be
very interesting because of their increased accuracy. However, a detailed comparison
of different operators is out of the scope of this article.

The least square least norm problems are solved using Krylov methods implemented
in the package
IterativeSolvers.jl\footnote{\url{https://github.com/JuliaMath/IterativeSolvers.jl},
version \texttt{v0.8.1}.} in Julia \cite{bezanson2017julia}.
To demonstrate that multiple solvers can be used, LSQR is applied in two space
dimensions and LSMR in three space dimensions. In these tests, LSMR has been
more performant than LSQR, i.e. similar errors of the potentials have been reached
in less runtime.

The source code for all numerical examples and figures (including
Figure~\ref{fig:grid_oscillations}) is published in \cite{ranocha2019discreteRepro}.

\subsection{Remaining Term \texorpdfstring{$\vec{r}$}{r} and Grid Oscillations}

As shown in Theorem~\ref{thm:SBP-helmholtz-impossible}, a discrete Helmholtz
decomposition \eqref{eq:general-helmholtz-decomposition} will in general have a
non-vanishing remaining term $\vec{r} \neq \vec{0}$, contrary to the continuous case.
As mentioned in Remark~\ref{rem:SBP-helmholtz-impossible}, the remainder
$\vec{r} = \vec{u} - \grad \vec{\phi} - \rot \vec{v}$ (in two space dimensions) seems to be linked to
some sort of grid oscillations.

Using the test problem of \cite{ahusborde2007primal}, given by
\begin{equation}
\label{eq:test-problem-ahusborde2007primal}
\begin{gathered}
  u(x_1, x_2) = \grad \phi + \rot v,
  \\
  \phi(x_1, x_2) = \sin(\pi(x_1 + x_2)),
  \quad
  v(x_1, x_2) = -\frac{1}{\pi} \sin(\pi x_1) \sin(\pi x_2),
\end{gathered}
\end{equation}
in the domain $[-1,1]^2$, the irrotational part $u_\mathrm{irr} = \grad \phi$
and the solenoidal part $u_\mathrm{sol} = \rot v$ can be computed exactly.
For this problem, the projection onto $\image \grad$ is performed at first,
in accordance with the conditions satisfied by the potentials $\phi, \psi$,
cf. Proposition~\ref{pro:continuous-projections-2-components}.

Applying the sixth order SBP operator of \cite{mattsson2004summation} on a grid
using $N_1 = N_2 = 60$ nodes in each coordinate direction yields the remainder
shown in Figure~\ref{fig:remainder_r}. While the components of the remainder
are not simple grid oscillations $\osc_1, \osc_2, \osc_{12}$, they are clearly
of a similar nature. Additionally, the amplitude of the remainder is approximately
four orders of magnitude smaller than that of the initial vector field $u$.
The results for other grid resolutions and orders of the operators are similar.

\begin{figure}[ht]
\centering
  \begin{subfigure}{0.49\textwidth}
    \centering
    \includegraphics[width=\textwidth]{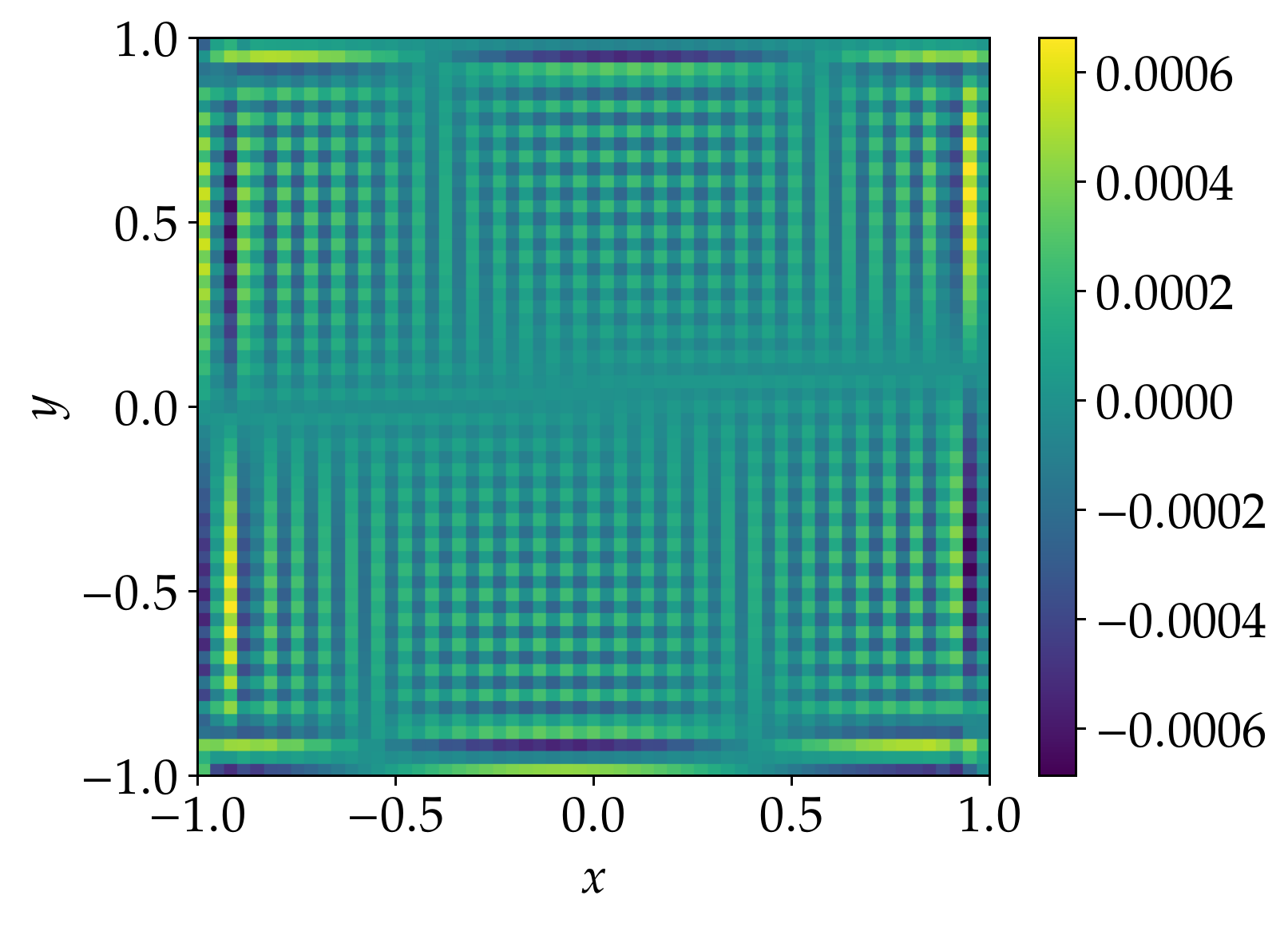}
    \caption{First component $\vec{r}_1$.}
  \end{subfigure}%
  \begin{subfigure}{0.49\textwidth}
    \centering
    \includegraphics[width=\textwidth]{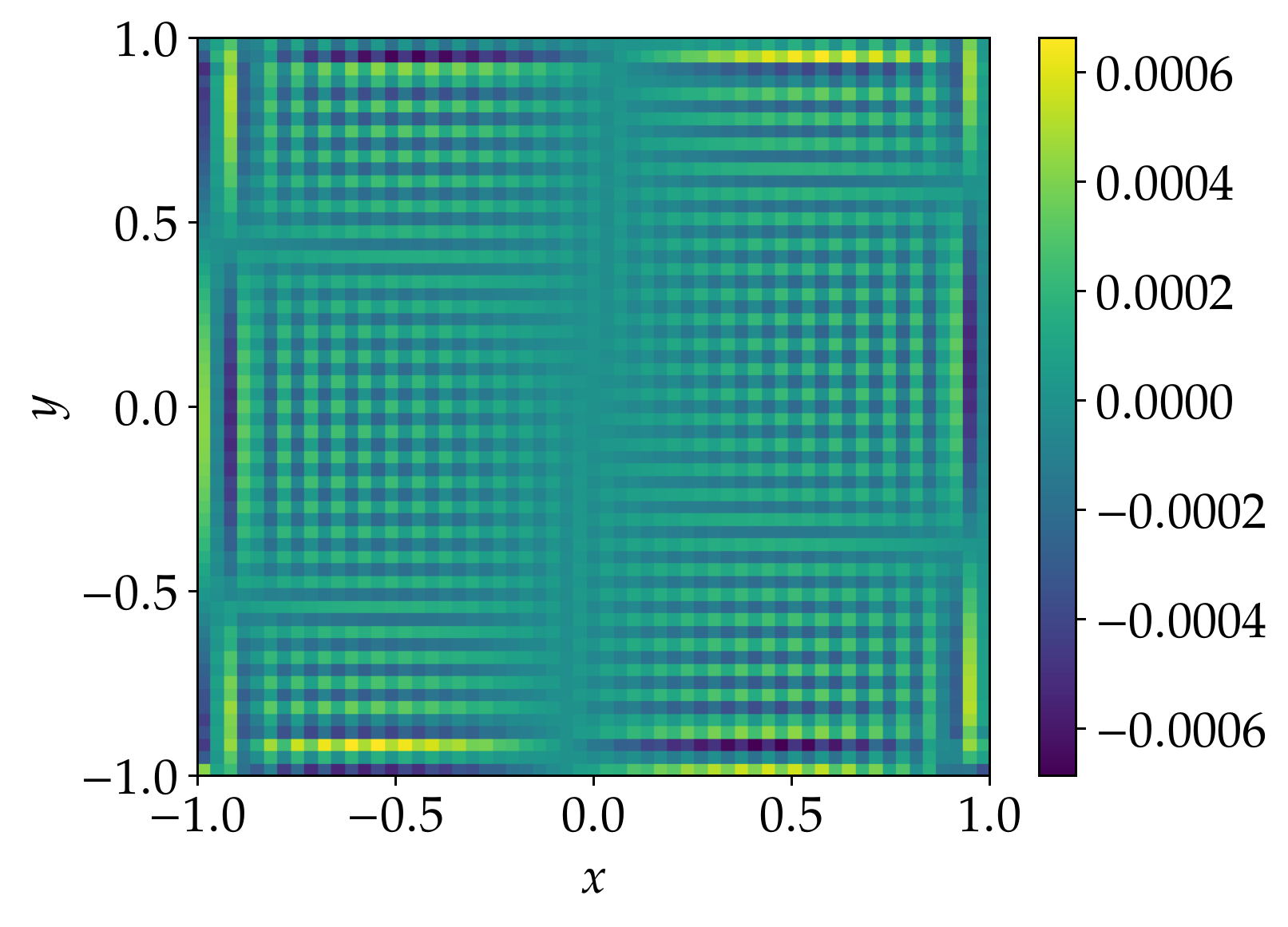}
    \caption{Second component $\vec{r}_2$.}
  \end{subfigure}
  \caption{Remainder $\vec{r} = \vec{u} - \grad \vec{\phi} - \rot \vec{v}$ of the discrete Helmholtz
           Hodge decomposition using the sixth order operator of
           \cite{mattsson2004summation} and $N_1 = N_2 = 60$ grid points per
           coordinate direction for the problem given by
           \eqref{eq:test-problem-ahusborde2007primal}.}
  \label{fig:remainder_r}
\end{figure}

Because of the scaling by the square root of the mass matrix described in
Section~\ref{sec:numerical-implementation}, the discrete projections are
(numerically) orthogonal with respect to the scalar product induced by the
mass matrix $M$. In this example,
\begin{equation}
\begin{aligned}
  \scp{\vec{u} - \grad \vec{\phi}}{\grad \vec{\phi}}_M &= \num{-2.15e-15},
  \\
  \scp{\vec{u} - \grad \vec{\phi} - \rot \vec{v}}{\rot \vec{v}}_M &= \num{9.26e-15}.
\end{aligned}
\end{equation}

\subsection{Convergence Tests in Two Space Dimensions}

Using the same setup \eqref{eq:test-problem-ahusborde2007primal} as in the
previous section, convergence tests using the second, fourth, sixth, and eighth
order operators of \cite{mattsson2004summation} are performed on $N_1 = N_2 = N$
nodes per coordinate direction.

\begin{figure}[tp]
\centering
  \begin{subfigure}{\textwidth}
    \centering
    \includegraphics[width=\textwidth]{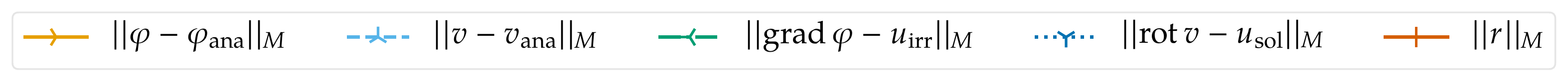}
  \end{subfigure}%
  \\
  \begin{subfigure}{0.49\textwidth}
    \centering
    \includegraphics[width=\textwidth]{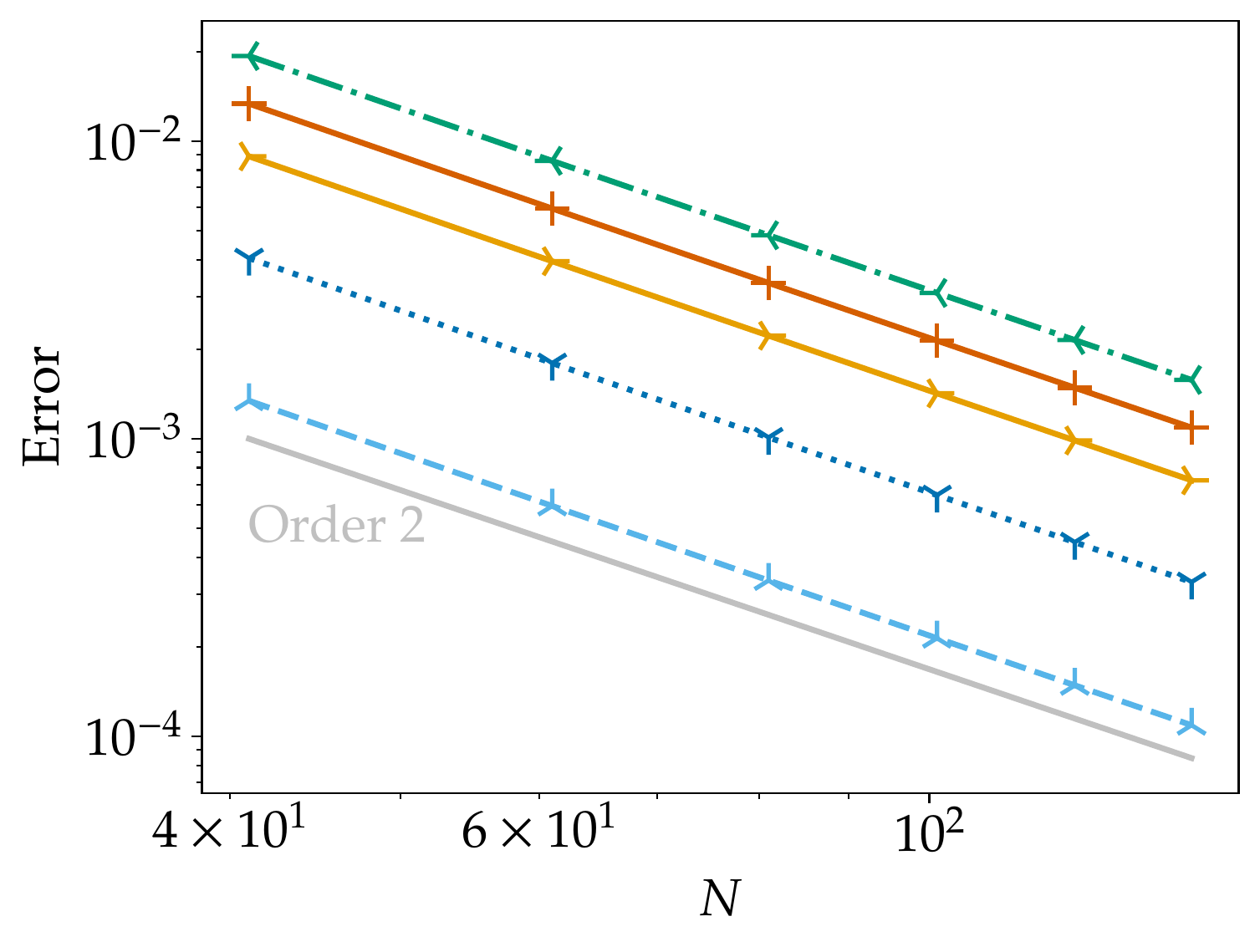}
    \caption{Interior order $2p = 2$.}
  \end{subfigure}%
  \begin{subfigure}{0.49\textwidth}
    \centering
    \includegraphics[width=\textwidth]{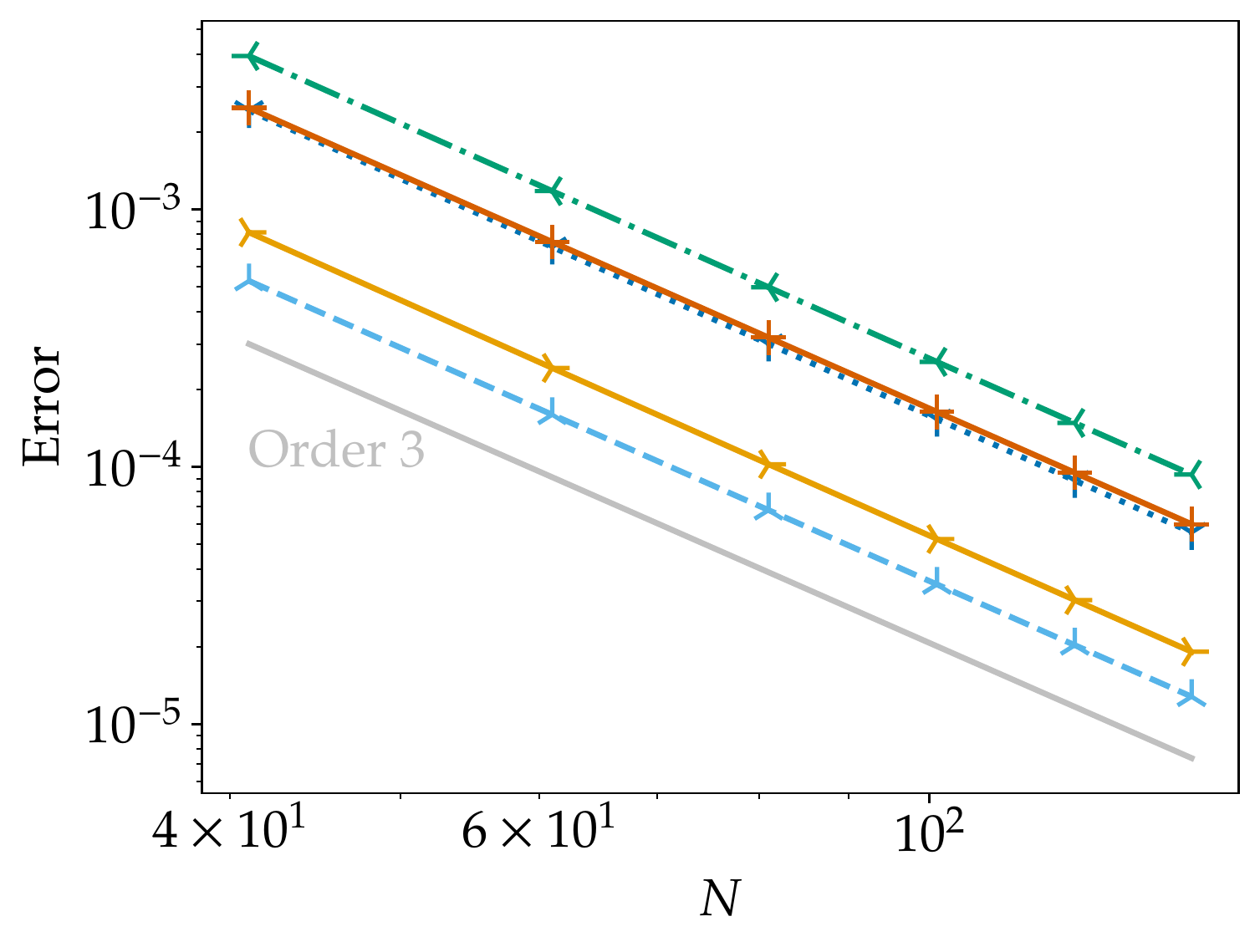}
    \caption{Interior order $2p = 4$.}
  \end{subfigure}
  \\
  \begin{subfigure}{0.49\textwidth}
    \centering
    \includegraphics[width=\textwidth]{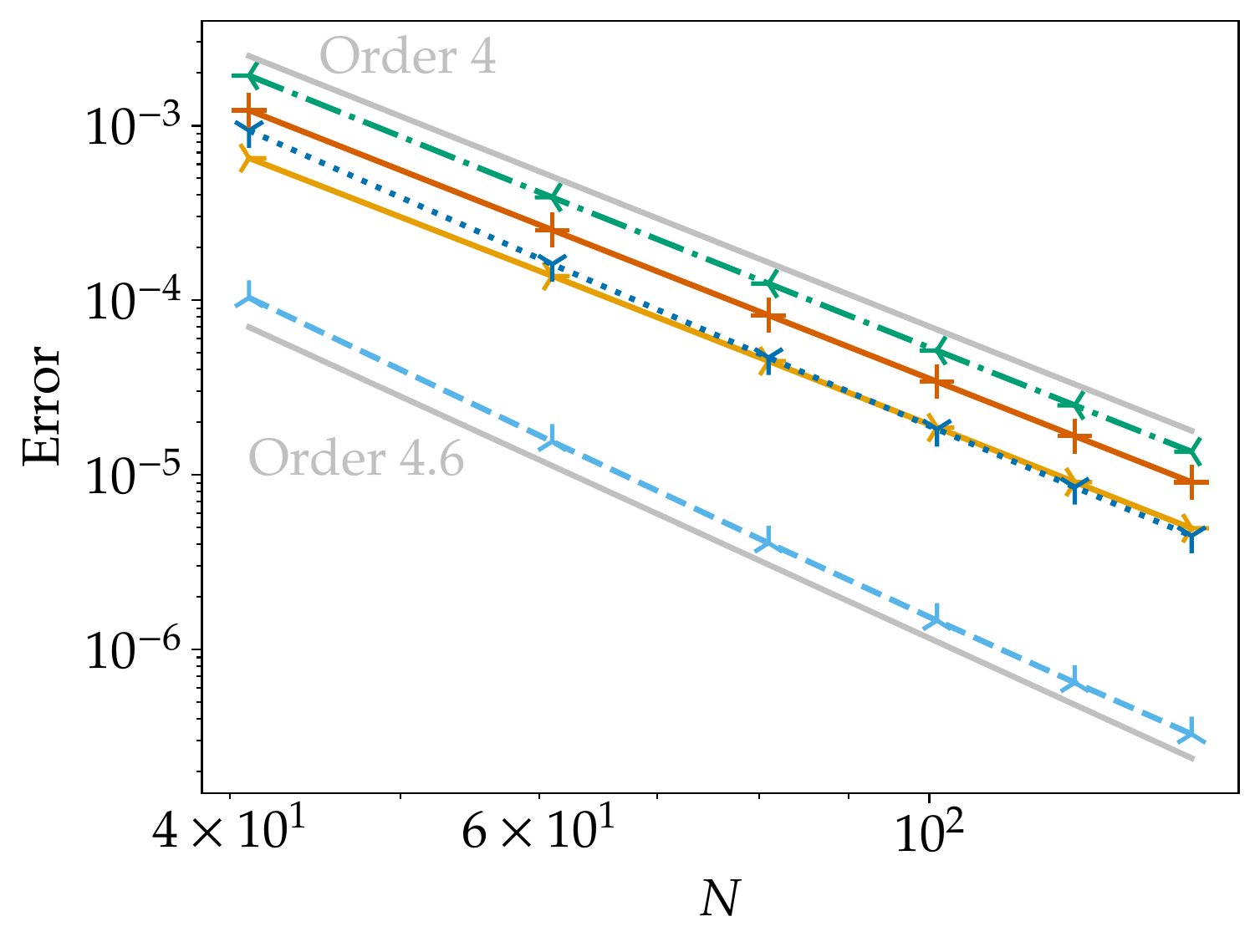}
    \caption{Interior order $2p = 6$.}
  \end{subfigure}%
  \begin{subfigure}{0.49\textwidth}
    \centering
    \includegraphics[width=\textwidth]{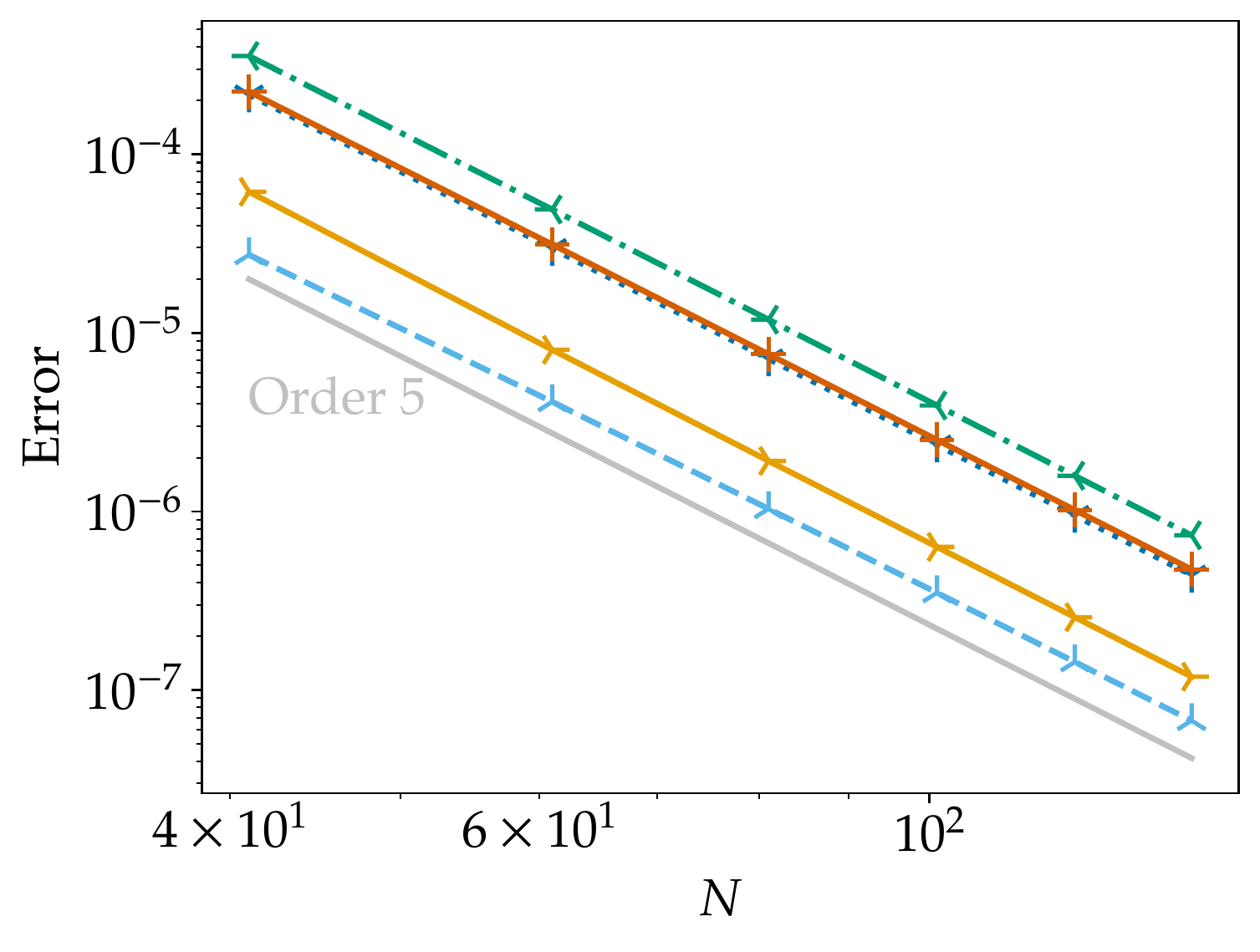}
    \caption{Interior order $2p = 8$.}
  \end{subfigure}
  \caption{Convergence diagrams of the discrete Helmholtz Hodge decomposition
           in two space dimensions
           using the SBP operators of \cite{mattsson2004summation} and $N_1 = N_2 = N$
           grid points per coordinate direction for the problem given by
           \eqref{eq:test-problem-ahusborde2007primal}.}
  \label{fig:convergence_2D}
\end{figure}

The results are visualised in Figure~\ref{fig:convergence_2D}. Both the potentials
$\vec{\phi}, \vec{v}$ and the irrotational/solenoidal parts $\grad \vec{\phi} = \vec{u}_\mathrm{irr}$,
$\rot \vec{v} = \vec{u}_\mathrm{sol}$ converge with an experimental order of accuracy of
$p+1$, as for suitable discretisations of some first order PDEs. The only exception
is given by the vector potential $\vec{v}$ for the
operator with interior order of accuracy $2p = 6$, which show an experimental
order of convergence of $4.6$ instead of $p+1 = 4$.

\subsection{Convergence Tests in Three Space Dimensions}

Here, another convergence test in three space dimensions is conducted. The problem
is given by
\begin{equation}
\label{eq:test-problem-3D}
\begin{gathered}
  u(x_1, x_2, x_3) = \grad \phi + \curl v,
  \\
  \phi(x_1, x_2, x_3) = \frac{1}{\pi} \sin(\pi x_1) \sin(\pi x_2) \sin(\pi x_3),
  \\
  v(x_1, x_2, x_3) = \frac{1}{\pi} \begin{pmatrix}
                                     \sin(\pi x_1) \cos(\pi x_2) \cos(\pi x_3) \\
                                     \cos(\pi x_1) \sin(\pi x_2) \cos(\pi x_3) \\
                                     -2 * \cos(\pi x_1) \cos(\pi x_2) \sin(\pi x_3)
                                   \end{pmatrix},
\end{gathered}
\end{equation}
in the domain $[-1,1]^3$. Again, the irrotational part
$u_\mathrm{irr} = \grad \phi$ and the solenoidal part
$u_\mathrm{sol} = \curl v$ can be computed exactly.
For this problem, the projection onto $\image \curl$ is performed at
first, in accordance with the boundary conditions $\phi|_{\partial \Omega} = 0$
and $v \cdot \nu |_{\partial \Omega} = 0$ satisfied by the potentials
$\phi, \psi$, cf. Proposition~\ref{pro:continuous-projections-2-components}.
As before, the second, fourth, sixth, and eighth order operators of
\cite{mattsson2004summation} are applied and $N_1 = N_2 = N_3 = N$
nodes per coordinate direction are used.

\begin{figure}[tp]
\centering
  \begin{subfigure}{\textwidth}
    \centering
    \includegraphics[width=\textwidth]{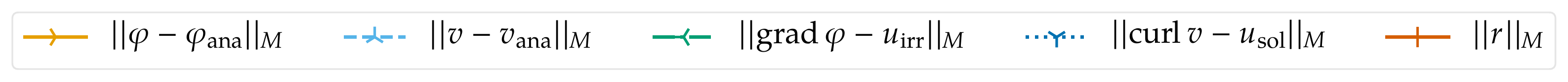}
  \end{subfigure}%
  \\
  \begin{subfigure}{0.49\textwidth}
    \centering
    \includegraphics[width=\textwidth]{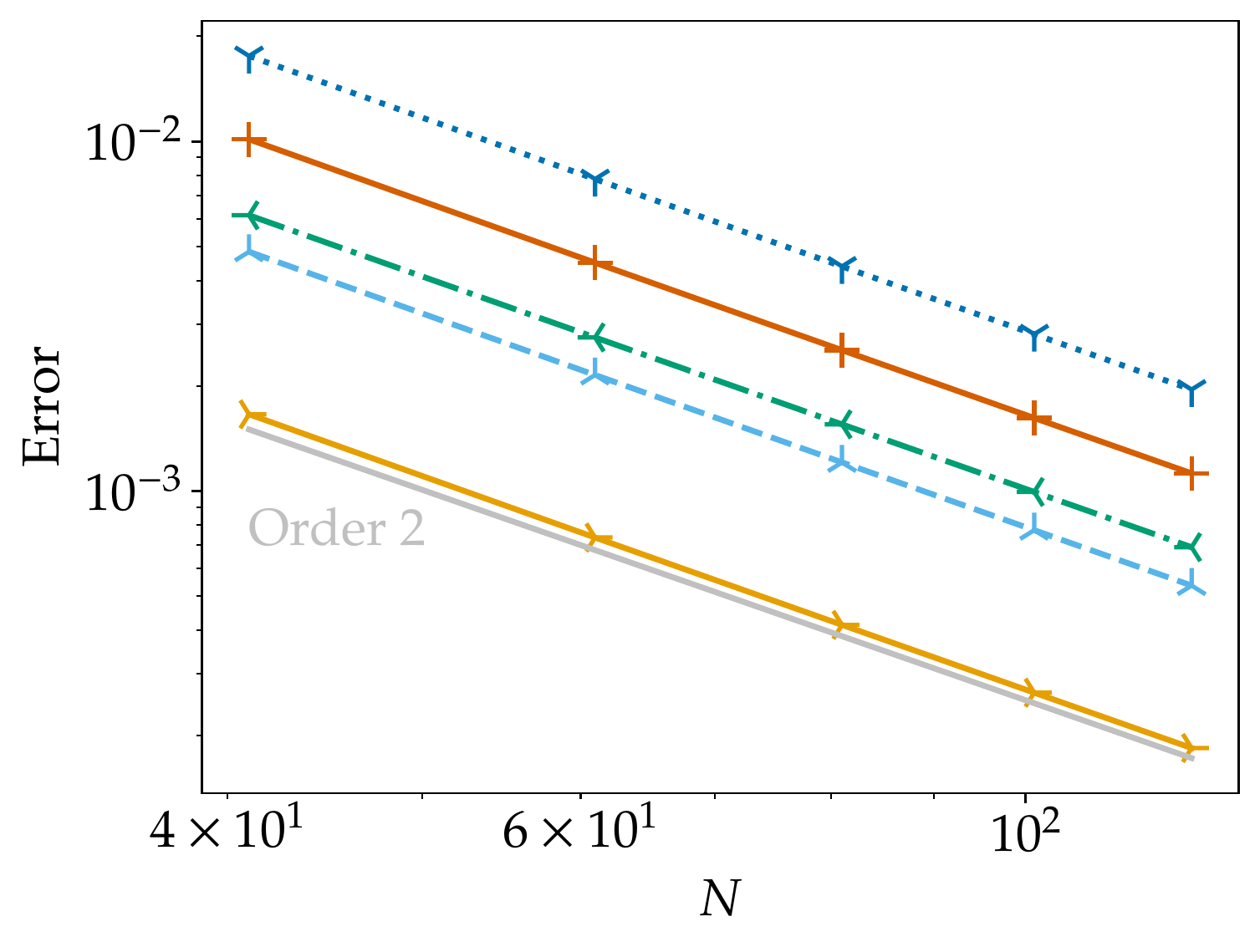}
    \caption{Interior order $2p = 2$.}
  \end{subfigure}%
  \begin{subfigure}{0.49\textwidth}
    \centering
    \includegraphics[width=\textwidth]{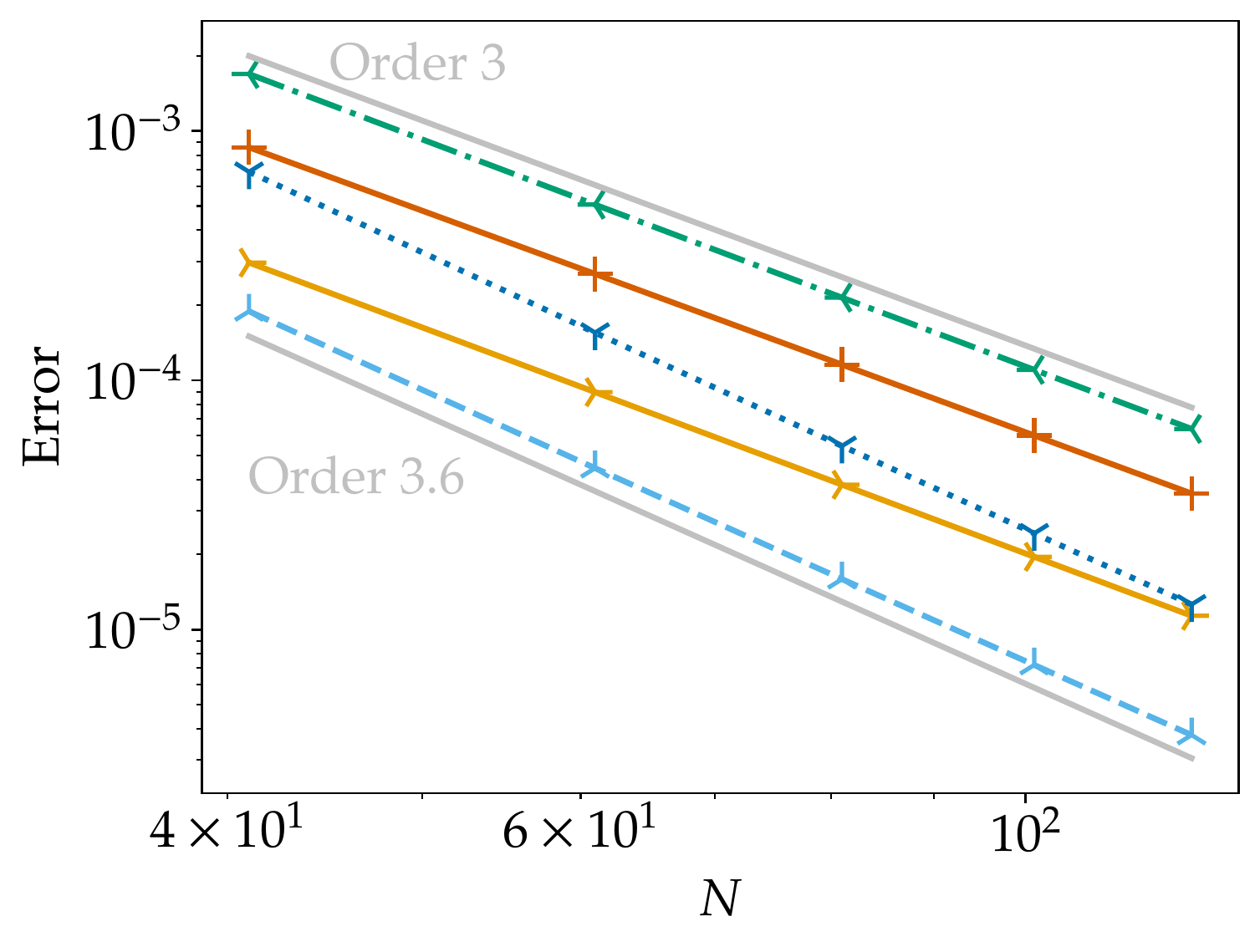}
    \caption{Interior order $2p = 4$.}
  \end{subfigure}
  \\
  \begin{subfigure}{0.49\textwidth}
    \centering
    \includegraphics[width=\textwidth]{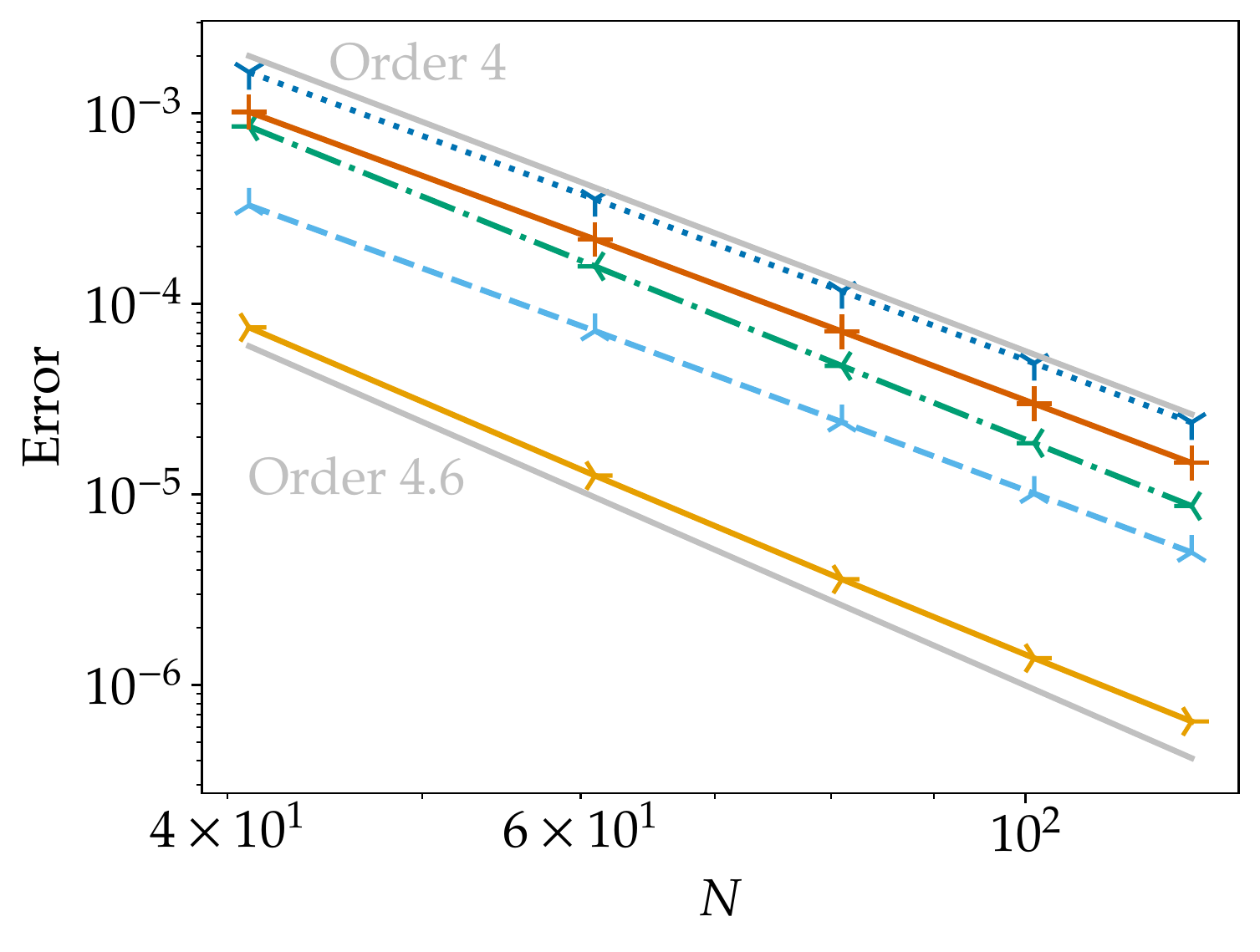}
    \caption{Interior order $2p = 6$.}
  \end{subfigure}%
  \begin{subfigure}{0.49\textwidth}
    \centering
    \includegraphics[width=\textwidth]{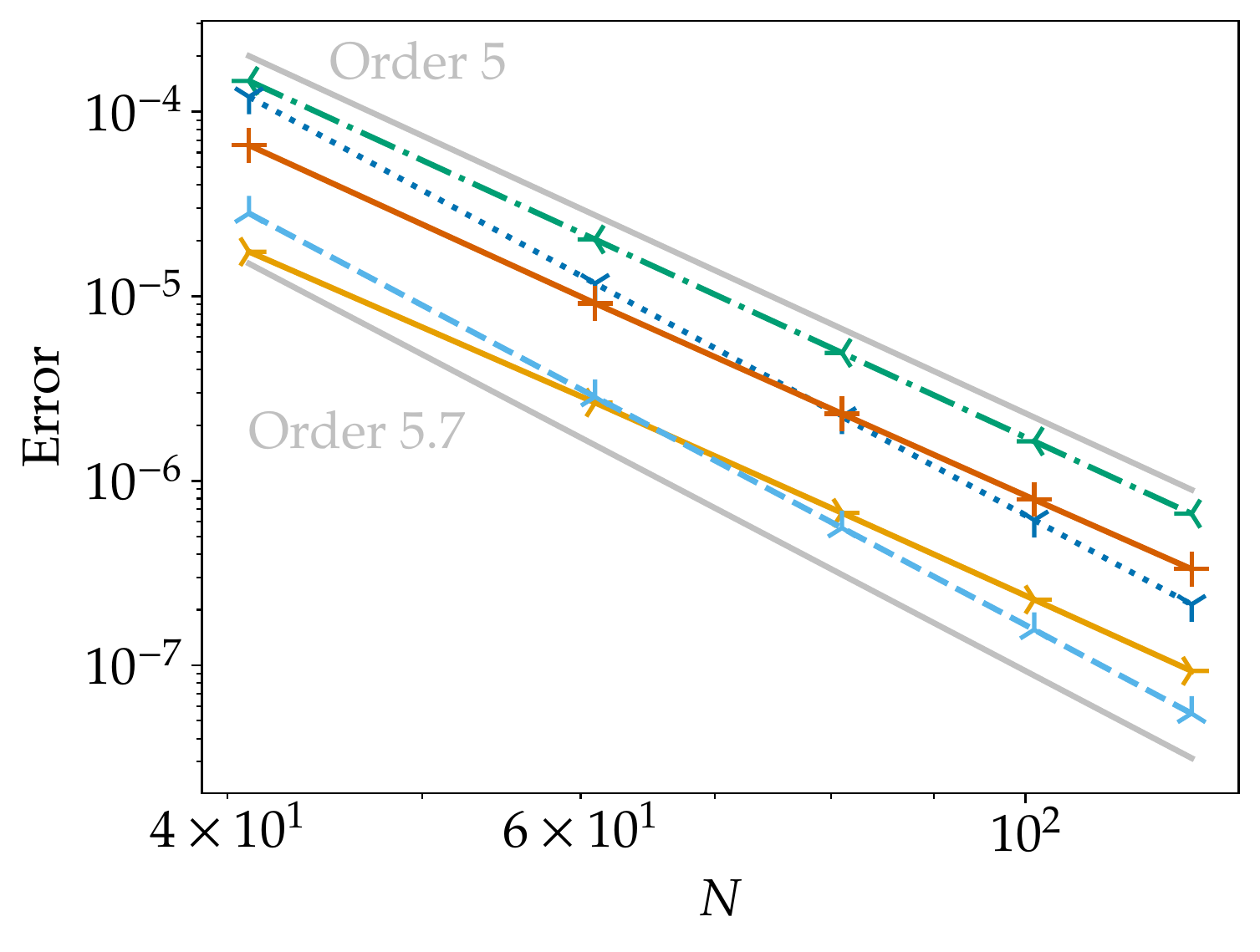}
    \caption{Interior order $2p = 8$.}
  \end{subfigure}
  \caption{Convergence diagrams of the discrete Helmholtz Hodge decomposition
           in three space dimensions using the SBP operators of
           \cite{mattsson2004summation} and $N_1 = N_2 = N_3 = N$
           grid points per coordinate direction for the problem given by
           \eqref{eq:test-problem-3D}.}
  \label{fig:convergence_3D}
\end{figure}

The results visualised in Figure~\ref{fig:convergence_3D} are similar to the
two-dimensional case considered before: The potentials and irrotational/solenoidal
components converge at least with an experimental order of accuracy $p+1$ for an
SBP operator with interior accuracy $2p$. Some potentials or parts converge with
an even higher order $\approx p + 1.5$ for the operators with $2p \in \set{4,6,8}$
in this test case.

\subsection{Analysis of MHD Wave Modes}

Here, the discrete Helmholtz Hodge decomposition will be applied to analyse
linear wave modes in ideal MHD. While the envisioned application in the future
concerns the analysis of numerical results obtained using SBP methods, analytical
fields will be used here to study the applicability of the methods developed in
this article.

Consider a magnetic field
\begin{equation}
  B(x_1,x_2,x_3)
  =
  \underbrace{\begin{pmatrix}
    0 \\
    0 \\
    1
  \end{pmatrix}}_{\mathllap{\text{background}}}
  +
  \underbrace{\begin{pmatrix}
    0 \\
    \epsilon_A \sin(k_1 x_1 + k_3 x_3) \\
    0
  \end{pmatrix}}_{\mathclap{\text{Alfvén}}}
  +
  \underbrace{\begin{pmatrix}
    0 \\
    0 \\
    - \epsilon_m \sin(k_1 x_1 + k_3 x_3)
  \end{pmatrix}}_{\mathclap{\text{magnetosonic}}},
\end{equation}
given as the sum of a background field, a transversal Alfvén mode, and a
longitudinal (fast) magnetosonic mode \cite[Chapter~23]{schnack2009lectures}.
Here, $\epsilon_A, \epsilon_m$ are the amplitudes of the linear waves and
$k = (k_1, 0, k_3)$ is the wave vector.

This magnetic field is discretised on a grid using $N_1 = N_2 = N_3 = N$ nodes
per coordinate direction in the box $\Omega = [-1,1]^3$. The current density
$\vec{j} = \curl \vec{B}$ is computed discretely and evaluated at the plane given by $x_3 = 0$.
There, the first and second component of $\vec{j}$ form the perpendicular current $\vec{j}^\perp$
in the $x_1$-$x_2$ plane. In the setup described above, the Alfvén mode is linked
to $\vec{j}^\perp_1$ and the magnetosonic mode yields $\vec{j}^\perp_2$.

Since the magnetosonic current is closed in the plane, the corresponding part of
$\vec{j}^\perp$ is divergence free. Since the Alfvén mode yields a current parallel
to the background field, the corresponding part of $\vec{j}^\perp$ is not solenoidal
but can be obtained via the Helmholtz Hodge decomposition
$\vec{j}^\perp = \grad \vec{\phi} + \rot \vec{v} + \vec{r}$,
where $\vec{r} \neq \vec{0}$ discretely in general.

While the Helmholtz Hodge decomposition is defined uniquely if $\Omega = \R^n$
is considered and can be used in plasma theory, there are some problems in bounded
domains because of the boundary effects/conditions. Numerically, discretisation
errors will also play a role.

\begin{figure}[tp]
\centering
  \includegraphics[width=\textwidth]{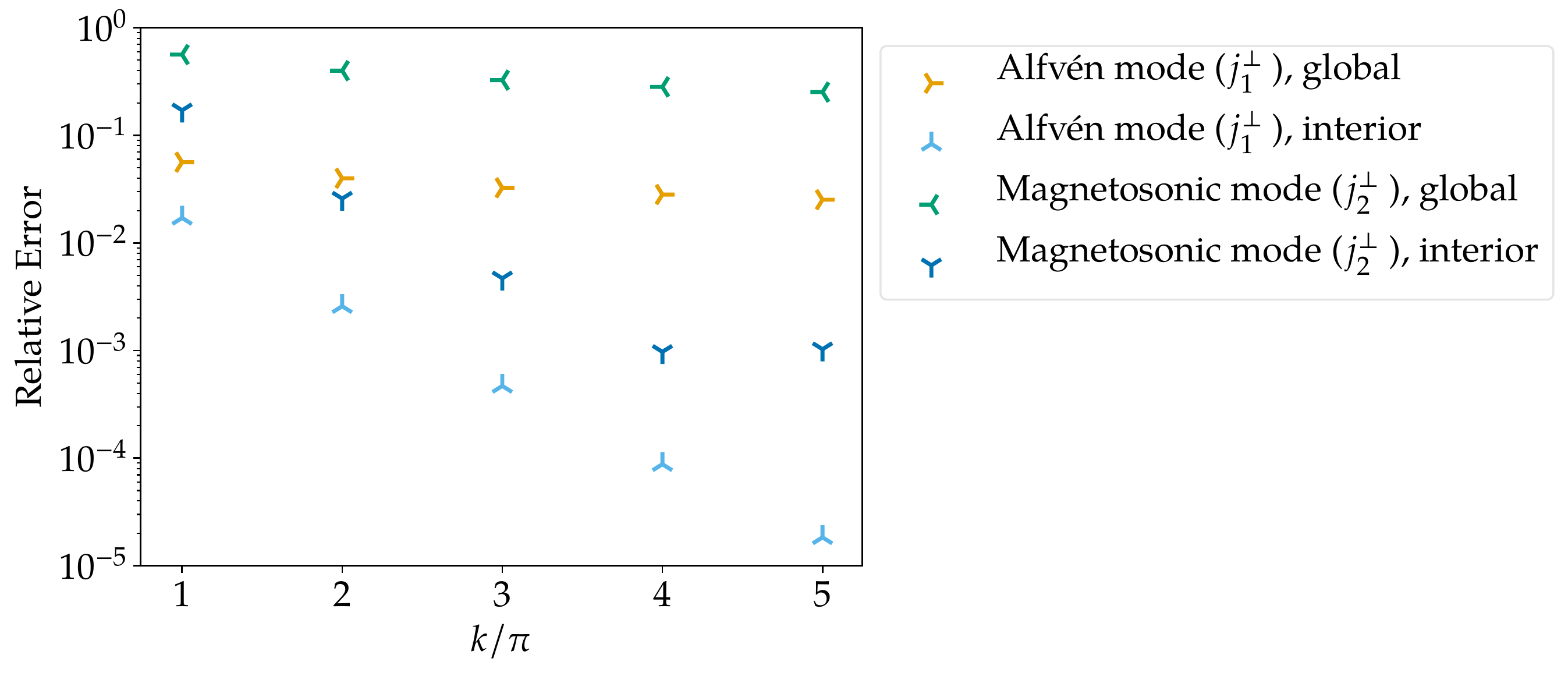}
  \caption{Errors of the wave mode components obtained via the discrete
           Helmholtz Hodge decomposition (projecting at first onto $\image \rot$)
           using the sixth order operator of \cite{mattsson2004summation} and
           $N = 101$ grid points per coordinate direction with parameters
           $k_1 = k_3 = k$, $\epsilon_A = 10^{-3}$, $\epsilon_m = 10^{-2}$.
           The global error is significantly bigger than the one in the interior
           (measured in the central quarter of the domain) because of disturbances
           at the boundaries.}
  \label{fig:MHD_modes__error_vs_k}
\end{figure}

The following observations have been made in this setup.
\begin{itemize}
  \item
  If one of the amplitudes $\epsilon_A, \epsilon_m$ vanishes and the projections
  are chosen in the correct order (projecting at first onto $\image \grad$ if
  $\epsilon_A \neq 0$ and onto $\image \rot$ if $\epsilon_m \neq 0$), $\grad \vec{\phi}$
  reproduces the current density of the Alfvén mode and $\rot \vec{v}$ that of the
  magnetosonic mode with only insignificant numerical artefacts.

  \item
  If the Alfvén and magnetosonic modes have amplitudes of the same order of
  magnitude, the order of the projections matters and disturbances are visible
  at the boundaries. Such disturbances occur even if one of the amplitudes
  vanishes but the projections are done in the wrong order.

  In Figure~\ref{fig:MHD_modes__error_vs_k}, errors of the wave mode components
  obtained via the discrete Helmholtz Hodge decomposition for such a test case
  are presented. Clearly, the global error is significantly bigger than the one
  in the interior (a square, centred in the middle of the domain, with one quarter
  of the total area).

  \item
  The disturbances from the boundaries are reduced if more waves are contained in
  the domain (e.g. if $k_1, k_3$ are increased while keeping the domain $\Omega$
  constant). For example, five waves in $\Omega$ have been sufficient in most
  numerical experiments to yield visually good results in the interior, cf.
  Figure~\ref{fig:MHD_modes__error_vs_k}.

  \item
  If one of the amplitudes is significantly bigger than the other one, e.g.
  because of phase mixing, the order of the projections should be chosen to
  match the order of the amplitudes to get better results. Thus, one should
  project at first onto $\image \grad$ if $\epsilon_A \gg \epsilon_m$ and
  at first onto $\image \rot$ if $\epsilon_m \gg \epsilon_A$.
  Otherwise, the smaller component is dominated by undesired contributions of
  the other one to its potential.

  \item
  If the ratio of the amplitudes is too big, contributions of the dominant
  mode can pollute the potential for the other mode significantly. The size of
  ratios that can be resolved on the grid depends on the number of grid nodes
  (increased resolution increases visible ratios) and the chosen SBP operator.
  For example, $\epsilon_A = 10^{-2}$ and $\epsilon_m = 10^{-4}$ yields acceptable
  results for the sixth order operator using $N = 61$ nodes. Choosing instead
  $\epsilon_m = 10^{-5}$, undesired contributions of the Alfvén mode to $(\rot \vec{v})_2$
  are an order of magnitude bigger than the desired contributions of the
  magnetosonic mode. This mode is visible again if the resolution is increased,
  e.g. to $N = 101$ grid points.
\end{itemize}
To sum up, the order of the projections has to be chosen depending on the given
data and one should experiment with both possibilities if there are no clear
hints concerning an advantageous choice. Additionally, there should be enough
waves in order to yield useful results that are not influenced too much by the
boundaries. Finally, the resolution should be high enough if big ratios of the
amplitudes are present.

If these conditions are satisfied, the discrete Helmholtz Hodge decomposition
can be applied successfully to analyse linear MHD wave modes. A typical plot
of the results for a ratio of wave amplitudes of $10^3$ is shown in
Figure~\ref{fig:MHD_modes__k1_pi__k3_pi__ea_1em2__ef_1em5}.

\begin{figure}[tp]
\centering
  \begin{subfigure}{\textwidth}
    \centering
    \includegraphics[width=\textwidth]{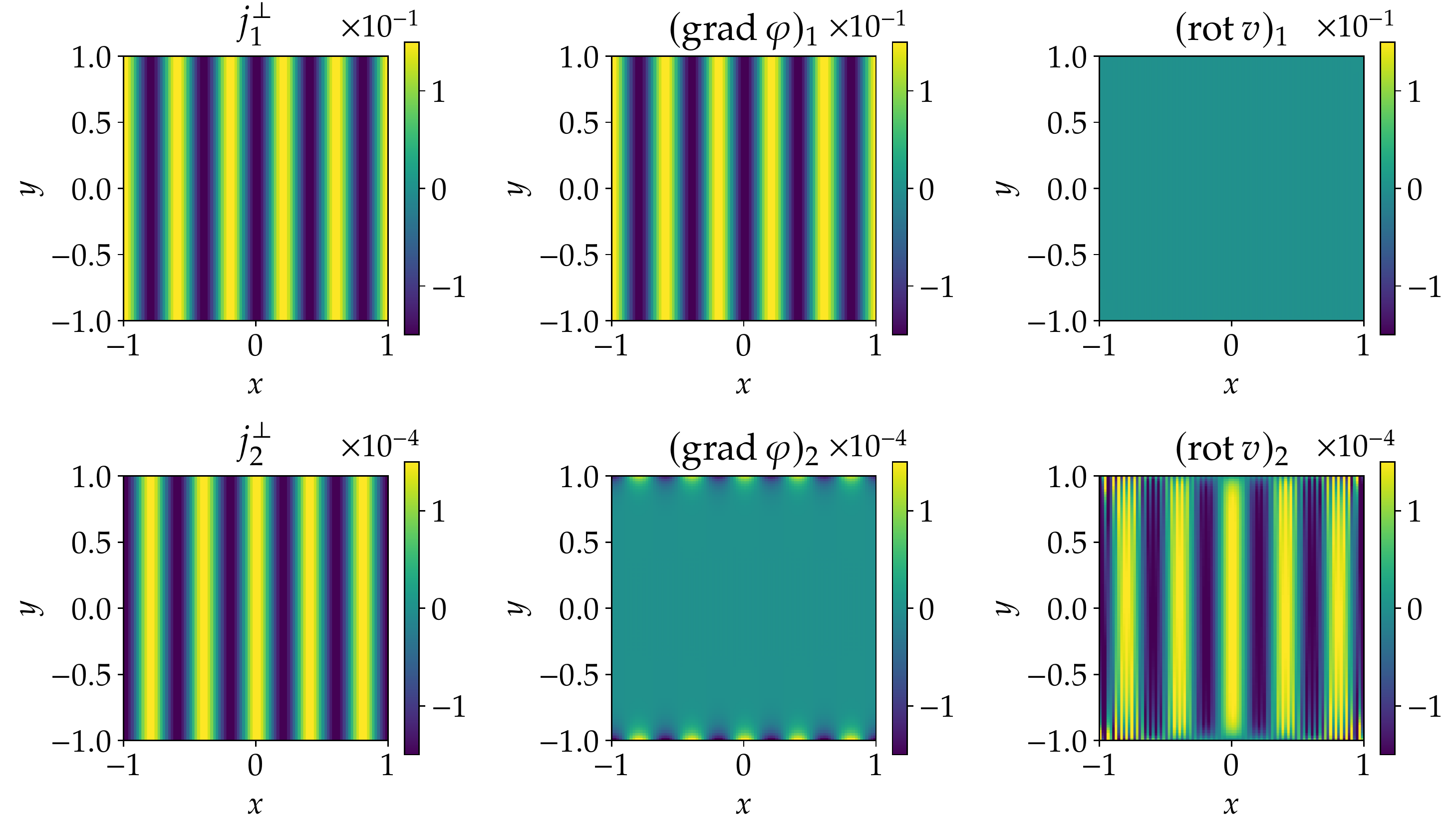}
    \caption{Projecting at first onto $\image \grad$.}
  \end{subfigure}%
  \\
  \begin{subfigure}{\textwidth}
    \centering
    \includegraphics[width=\textwidth]{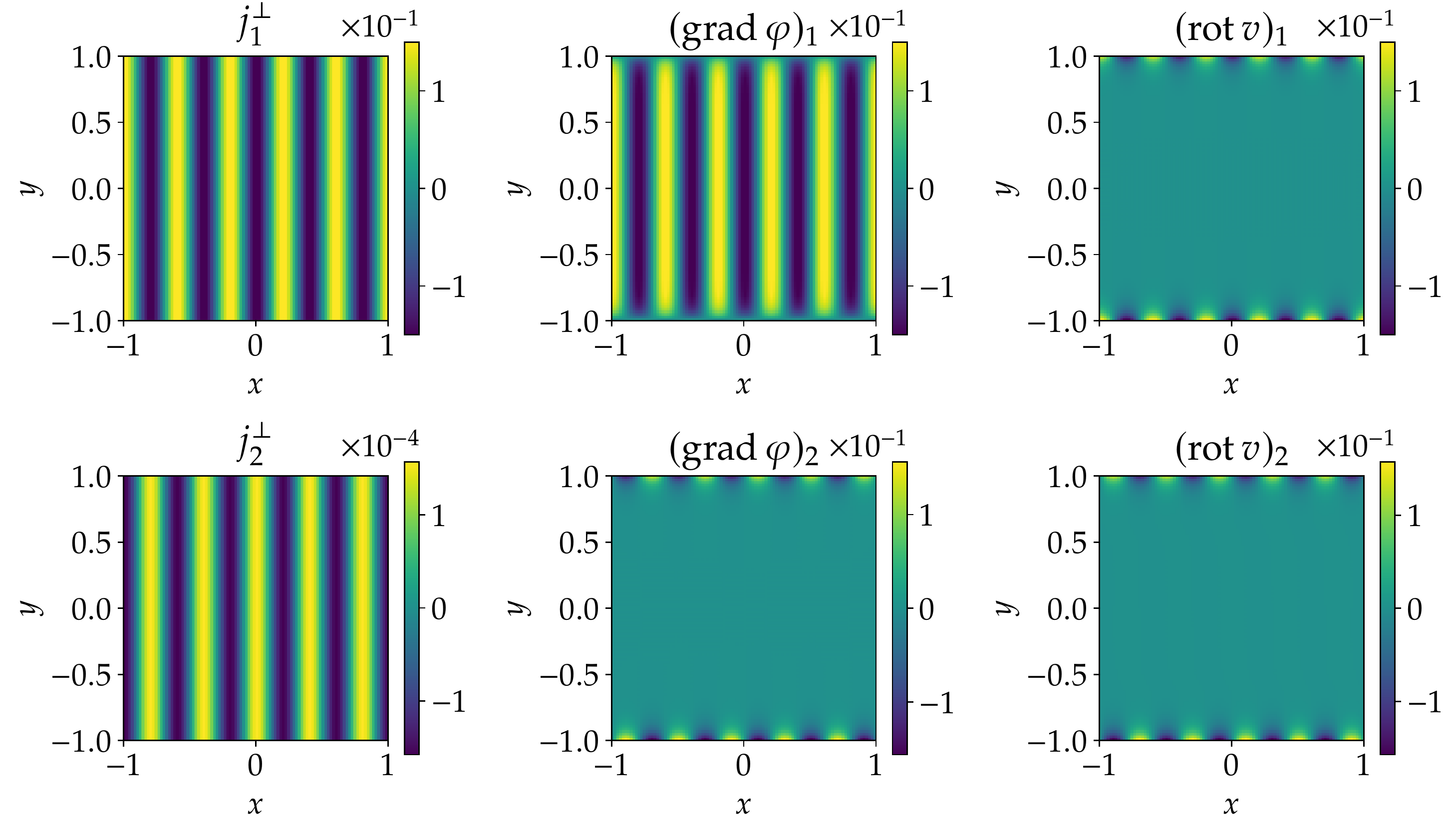}
    \caption{Projecting at first onto $\image \rot$.}
  \end{subfigure}
  \caption{Discrete current density $\vec{j}^\perp$ and its Helmholtz Hodge decomposition
           using the sixth order operator of \cite{mattsson2004summation} and
           $N = 101$ grid points per coordinate direction with parameters
           $k_1 = k_3 = 5\pi$, $\epsilon_A = 10^{-2}$, $\epsilon_m = 10^{-5}$.}
  \label{fig:MHD_modes__k1_pi__k3_pi__ea_1em2__ef_1em5}
\end{figure}

Characterising the chosen variant of the Helmholtz Hodge
decomposition using boundary conditions is not necessarily advantageous
in this example. Indeed, the ratio of the amplitudes of the Alfvén
and magnetosonic waves does not influence the different types of
homogeneous boundary conditions
given in Propositions~\ref{pro:continuous-HHD-BCs} and
\ref{pro:continuous-projections-2-components} (as long as both
amplitudes $\epsilon_A$ and $\epsilon_m$ do not vanish). Hence, basing
a choice of a variant of the Helmholtz Hodge decomposition solely on
boundary conditions, one would not expect to see a qualitative difference
between the behaviour of the variants for $\epsilon_A \gg \epsilon_m$ vs.
$\epsilon_A \ll \epsilon_m$. However, such a qualitative difference can
be clearly observed in practice, as can be seen by comparing
Figure~\ref{fig:MHD_modes__k1_pi__k3_pi__ea_1em2__ef_1em5}
($\epsilon_A \gg \epsilon_m$) to
Figure~\ref{fig:MHD_modes__k1_pi__k3_pi__ea_1em5__ef_1em2}
($\epsilon_A \ll \epsilon_m$).
In the former case, projecting at first onto $\image \grad$ is
advantageous while projecting at first onto $\image \rot$ is better
in the latter case.
This behaviour is in accordance with the discussion above based on the
interpretation of the variants of the Helmholtz Hodge decomposition
as projections using different orders.
\begin{figure}[tp]
\centering
  \begin{subfigure}{\textwidth}
    \centering
    \includegraphics[width=\textwidth]{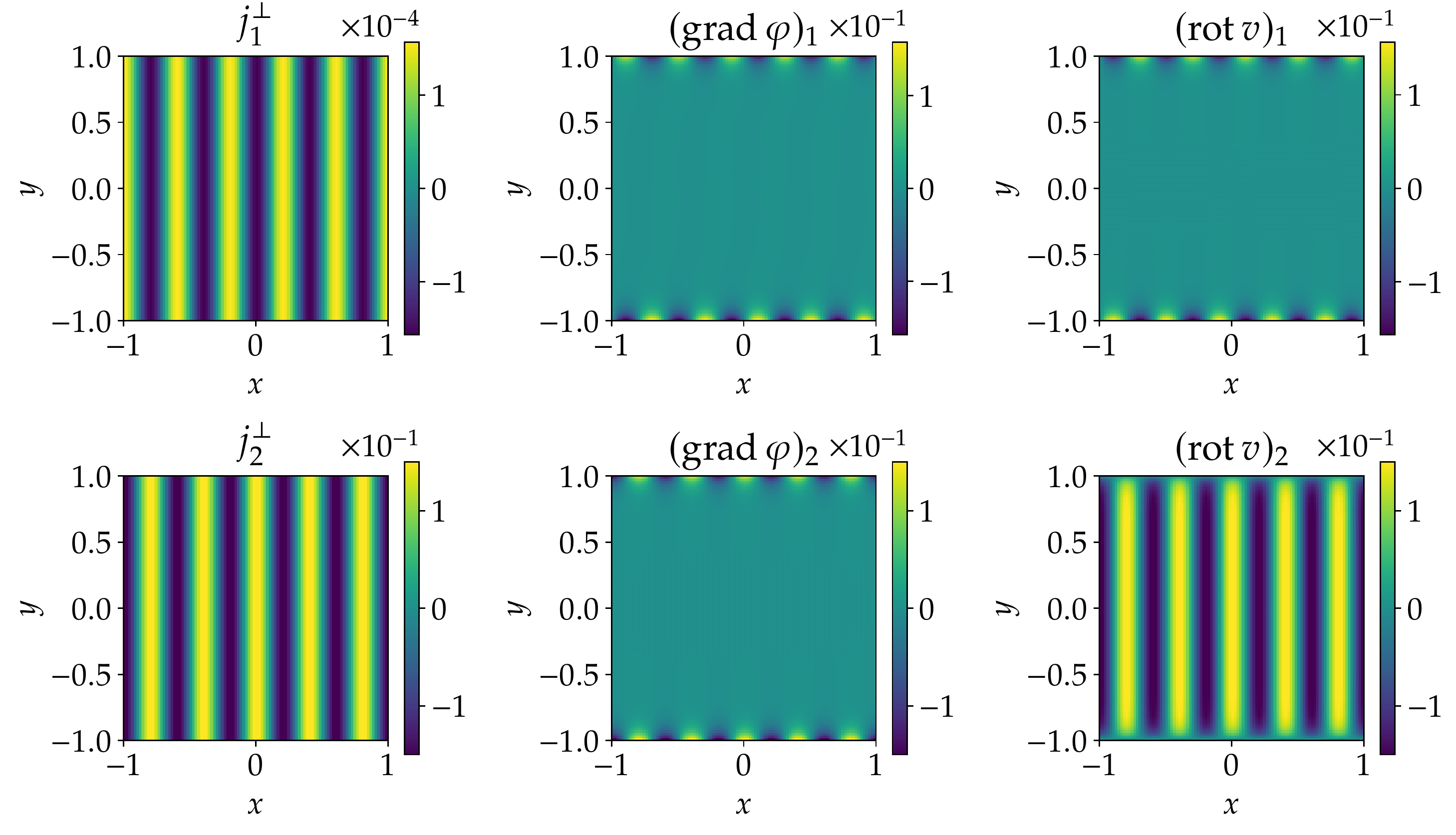}
    \caption{Projecting at first onto $\image \grad$.}
  \end{subfigure}%
  \\
  \begin{subfigure}{\textwidth}
    \centering
    \includegraphics[width=\textwidth]{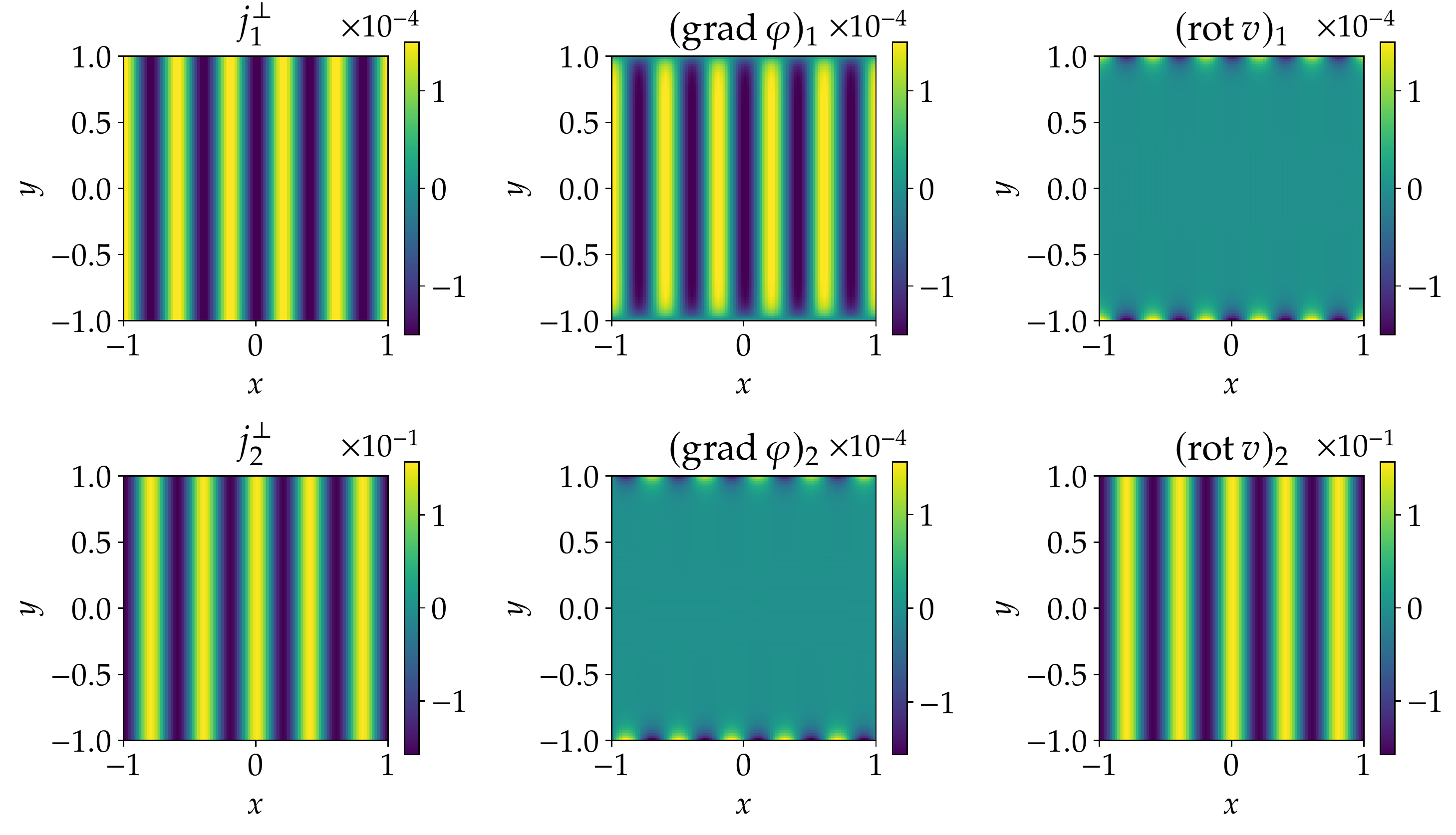}
    \caption{Projecting at first onto $\image \rot$.}
  \end{subfigure}
  \caption{Discrete current density $\vec{j}^\perp$ and its Helmholtz Hodge decomposition
           using the sixth order operator of \cite{mattsson2004summation} and
           $N = 101$ grid points per coordinate direction with parameters
           $k_1 = k_3 = 5\pi$, $\epsilon_A = 10^{-5}$, $\epsilon_m = 10^{-2}$.}
  \label{fig:MHD_modes__k1_pi__k3_pi__ea_1em5__ef_1em2}
\end{figure}

\section{Summary and Discussion}
\label{sec:summary}

In this article, discrete variants of classical results from vector calculus
for finite difference summation by parts operators have been investigated. Firstly,
it has been proven that discrete variants of the classical existence theorems for
scalar/vector potentials of curl/divergence free vector fields cannot hold discretely,
cf. Theorems~\ref{thm:SBP-scalar-potential-2D}, \ref{thm:SBP-vector-potential-2D},
\ref{thm:SBP-scalar-potential-3D}, and \ref{thm:SBP-vector-potential-3D}, basically
because of the finite dimensionality of the discrete functions spaces and the
presence of certain types of grid oscillations.

Based on these results, it has been shown that discrete Helmholtz Hodge
decompositions $\vec{u} = \grad \vec{\phi} + \curl \vec{v} + \vec{r}$ of a given vector field $\vec{u}$ into
an irrotational component $\vec{u}_\mathrm{irr} = \grad \vec{\phi}$ and a solenoidal part
$\vec{u}_\mathrm{sol} = \curl \vec{v}$ will in general have a non-vanishing remainder $\vec{r}$,
contrary to the continuous case, cf. Section~\ref{sec:helmholtz}. This remainder
$\vec{r} \neq \vec{0}$ is associated to certain types of grid oscillations, as supported by
theoretical insights and numerical experiments in Section~\ref{sec:numerical-examples}.
There, applications to the analysis of MHD wave modes are presented and discussed
additionally.

Classically, different variants of the Helmholtz Hodge decomposition
in bounded domains are most often presented with an emphasis on boundary
conditions. Taking another point of view, these variants can also be
interpreted as results of two orthogonal projections in a Hilbert space,
i.e. as least squares problems. Since the images/ranges of these projections
are not orthogonal, the projections do not commute and their order
matters, resulting in different variants of the decomposition.
At the continuous level, these correspond to the different types of
boundary/secondary conditions for the potentials in the associated
normal equations of the least squares problems, which are elliptic PDEs.
Here, computing the least norm least squares solution via iterative methods
has been proposed and applied successfully to compute discrete Helmholtz Hodge
decompositions. Using advanced iterative solvers such as LSQR or LSMR, the
solution of the least norm least squares problem for the projections
corresponds to the solution of the underlying elliptic PDEs via the
related iterative methods, e.g. CG for LSQR and MINRES for LSMR. However,
LSQR/LSMR are known to have some advantageous properties compared
to the application of CG/MINRES to the discretised elliptic PDEs.

The basic argument for the impossibility of a discrete Poincaré lemma (existence
of scalar/vector potentials for irrotational/solenoidal vector fields) uses the
finite dimension of the discrete function spaces and the collocation approach.
If staggered grids are used instead, these arguments do not hold in the same form
and potentials exist for some (low order) operators, e.g. in \cite{silberman2019numerical}
or for the mimetic operators of \cite{hyman1999orthogonal}. Hence, it will be
interesting to consider staggered grid SBP operators in this context, cf.
\cite{oreilly2017energy,mattsson2018compatible,gao2019sbp}.

At the continuous level, there are two widespread versions of the Helmholtz
Hodge decomposition, characterised either by the choice of boundary
conditions for the potentials or by the order of projections onto
$\image \grad = \kernel \curl$ and $\image \curl = \kernel \div$.
Since there are different possibilities to impose boundary conditions
and the relations of the images and kernels do not hold discretely, there
are several other discrete variants. In this article, orthogonal projections
onto $\image \grad$ and $\image \curl$ have been considered, corresponding
to a certain weak imposition of boundary conditions for the potentials.
Projecting instead onto $\kernel \curl$ and $\kernel \div$ is another
option that seems to be viable and will be studied in the future.

The iterative methods used for the orthogonal projections in this article are
equivalent to the application of certain methods such as CG or MINRES to the
associated discrete normal equations in exact arithmetic. At the continuous
level, these normal equations are elliptic second order problems. For example,
the scalar potential is associated to a Neumann problem. These elliptic PDEs
could also be solved discretely using (compatible) narrow stencil operators
while the discrete normal problems are associated to wide stencil operators.
There are also other approaches to approximate Helmholtz Hodge decompositions
discretely, e.g. \cite{angot2013fast,ahusborde2014discrete}. While a detailed
comparison of all these approaches is out of the scope of this article, it would
be interesting for the community and physicists interested in the application
of discrete Helmholtz Hodge decompositions. Of course, the advantages and drawbacks
of different iterative solvers and preconditioners should be considered
for such a detailed comparison as well.

\printbibliography

\end{document}